\documentclass[1996/10/24]{amsart}

\allowdisplaybreaks[3]

\usepackage{amsfonts}
\usepackage{epsf,latexsym,amsfonts,amsbsy,mathrsfs}
\usepackage{amsmath, amssymb, amsthm,amscd,amsxtra}
\usepackage[]{graphicx,subfigure}
\usepackage{epstopdf}
\usepackage{float}
\usepackage{multirow}
\usepackage{bbm}
\usepackage{stmaryrd}
\usepackage{epsfig}
\usepackage{tikz}
\usepackage{dsfont}
\usetikzlibrary{positioning}
\usepackage{xcolor}

\usepackage{amssymb,stmaryrd,color}
\usepackage{lipsum}
\usepackage{amsfonts}
\usepackage{graphicx}
\usepackage{epstopdf}
\usepackage{algorithmic}
\usepackage{bm}
\usepackage{multirow}

\ifpdf
  \DeclareGraphicsExtensions{.eps,.pdf,.png,.jpg}
\else
  \DeclareGraphicsExtensions{.eps}
\fi

\newtheorem{theorem}{Theorem}[section]
\newtheorem{lemma}[theorem]{Lemma}

\theoremstyle{definition}

\newtheorem{remark}[theorem]{Remark}

\newtheorem{corollary}[theorem]{Corollary}

\numberwithin{equation}{section}

\newcommand{\ignore}[1]{}

\newcommand{\vertiii}[1]{{\left\vert\kern-0.25ex\left\vert\kern-0.25ex\left\vert #1
    \right\vert\kern-0.25ex\right\vert\kern-0.25ex\right\vert}}

\begin{document}

\title[An Interface Green's Function Framework]{An Interface Green's Function Framework for Complete Discrete
$W^{1,\infty}$ Analysis of Discontinuous Galerkin Methods}

\author{Haitao Leng}
\address{School of Mathematics and Information Science, Guangzhou University, Guangzhou 510006, Guangdong, China}
\email{htleng@m.scnu.edu.cn}

\author{Weifeng Qiu}
\address{Department of Mathematics, City University of Hong Kong, Hong Kong, China}
\email{weifeqiu@cityu.edu.hk}

\thanks{Weifeng Qiu is corresponding author.}



\begin{abstract}
Pointwise error analysis of discontinuous Galerkin (DG) methods for the Poisson equation has received considerable attention over the past two decades. However, on convex polyhedral domains, existing analyses can only establish
\emph{nearly} (with a logarithmic factor) optimal error estimates in the $L^{\infty}$ norm and optimal error estimates in the broken $W^{1,\infty}$ seminorm for several DG methods. Since numerical solutions obtained by DG methods may have jumps
across mesh interfaces, and the broken $W^{1,\infty}$ seminorm does not control these jumps, a complete discrete
$W^{1,\infty}$ theory for DG methods for the Poisson equation on convex polyhedral domains has remained unavailable.
Here, by a complete discrete $W^{1,\infty}$ theory, we mean that the maximum-norm of the broken gradient and
the maximum norm of the interface jumps are estimated simultaneously.

In this paper, we develop an interface Green's function framework for the complete discrete $W^{1,\infty}$ analysis of DG methods on convex polyhedral domains. The proposed framework introduces new interface Green's functions that represent the jumps of numerical solutions of DG methods across mesh interfaces. Its central analytical ingredient is a new local energy estimate for these Green's functions, obtained by exploiting a cancellation between neighboring discrete delta functions. This local energy estimate differs fundamentally from existing Green's function estimates and enables us to derive maximum-norm estimates for the jumps of numerical solutions across mesh interfaces without
introducing additional logarithmic factors.

Via the proposed framework, we establish optimal error estimates in the complete discrete $W^{1,\infty}$ norm for the symmetric interior penalty DG method, the hybrid mixed DG method, and two representative hybridizable DG methods
for the Poisson equation in convex polyhedral domains.
In particular, we derive optimal maximum-norm estimates for interface jumps, numerical fluxes, and broken gradients.
We further extend the framework to the $C^0$ interior penalty method for the biharmonic equation on convex polygonal
domains and obtain maximum-norm estimates for the jumps of the normal derivative. The interface Green's function
framework is sufficiently general to be applicable to a broad class of nonconforming methods and is valid on general
polyhedral meshes.
\end{abstract}


\maketitle

\section{Introduction}

Pointwise approximation has long been one of the central topics in the numerical analysis of partial differential equations (PDEs). Compared with energy-norm estimates, pointwise estimates provide considerably more detailed information on the local behavior of numerical solutions and play a crucial role in adaptive computation \cite{dg2012, nssv2006}, optimal control \cite{dh2007, lc2026}, singular source problems \cite{ddo2020}, and the numerical approximation of nonlinear PDEs \cite{d2006, dd1975}. Since the pioneering works of Schatz and Wahlbin \cite{sw1978,sw1995}, the theory of maximum-norm estimates has developed into a mature branch of finite element analysis.

For conforming finite element methods (FEM) for the Poisson equation, the maximum-norm theory is now 
well understood. Classical works first established $L^\infty$- and $W^{1,\infty}$-error estimates for smooth \cite{sw1995} and polygonal domains \cite{sw1978}. Subsequent research expanded this theory to graded meshes \cite{dlsw2012} and polyhedral domains \cite{drs2024, glrs2009}. Alongside these developments, weak discrete maximum principles were also established \cite{s1980, ll2020}. These developments have provided a comprehensive understanding of pointwise approximations for FEM.

Over the last two decades, DG methods have become one of the most successful classes of finite element methods owing to their flexibility on general meshes, local conservation, suitability for $hp$-adaptivity, and excellent parallel scalability; see, for example, \cite{c2002,de2011}. Numerous DG methods have subsequently been developed for second-order elliptic problems, including interior penalty discontinuous Galerkin (IPDG) methods \cite{abcm2002, mss2017}, local discontinuous Galerkin (LDG) methods \cite{c2006, ccps2000, ckps2001}, hybrid mixed DG methods \cite{es2010}, hybridizable discontinuous Galerkin (HDG) methods \cite{cgl2009, cgnps2011, lc2022, qs2016}, and hybrid high-order methods \cite{be2018, cel2019}. 

The pointwise error analysis for DG methods for the Poisson equation has also attracted considerable attention.
Chen and Chen \cite{cc2005} established the first maximum-norm estimates for IPDG methods. Shortly afterwards, Chen \cite{c2005} and Guzm\'an \cite{g2006} obtained corresponding results for LDG methods. However, the estimates in \cite{cc2005,c2005,g2006} are valid only on smooth domains.
More recently, Leng and Chen \cite{lc2026} developed a unified analysis of \emph{nearly} optimal error estimates
in the $L^{\infty}$ norm and optimal error estimates in the broken $W^{1,\infty}$ seminorm for a class of IP-HDG methods on convex polyhedral domains. We denote by $u_h^{{\rm nm}}$ the numerical solution of any DG methods
which can be IPDG, LDG, IP-HDG, etc. Roughly speaking, we have that for any function $\chi$ in the discrete space,
\begin{align}
\|u-u_h^{{\rm nm}}\|_{L^{\infty}(\Omega)}\leq& C|\ln h|^{\beta}\left(\|u-\chi\|_{L^{\infty}(\Omega)}+h\|\nabla(u-\chi)\|_{L_h^{\infty}(\Omega)}\right),\label{nm-mn:1}\\
\|\nabla(u-u_h^{{\rm nm}})\|_{L_h^{\infty}(\Omega)}=&\max_{K\in\mathcal T_h} \|\nabla (u - u_{h}^{\rm nm}) \|_{L^\infty(K)}\leq C\|\nabla(u-\chi)\|_{L_h^{\infty}(\Omega)},\label{nm-mn:2}
\end{align}
where $\mathcal{T}_h$ is a mesh of the domain $\Omega$, $\beta$ is a positive constant if $\Omega$ is a convex polyhedron, and $\beta = 0$ if $\Omega$
is a smooth domain with polynomial degrees greater than one. Despite these developments, a critical gap remains in the existing theory.

In contrast to FEM, for DG methods,
$\|\nabla(u-u_h^{{\rm nm}})\|_{L_h^\infty(\Omega)}$
is merely the broken $W^{1,\infty}$ seminorm. It contains no information on discontinuities of $u_{h}^{{\rm nm}}$
across mesh interfaces. Consequently, existing pointwise analyses provide no direct control over interface jumps
of $u_{h}^{{\rm nm}}$ in the $L^{\infty}$ norm. Since these quantities are intrinsic to DG methods for the Poisson
equation, it is natural to seek a complete discrete $W^{1,\infty}$ theory in which the broken gradient and interface
jumps are estimated simultaneously.

The primary objective of this paper is to establish such a theory.

To this end, we introduce the discrete $W^{1,\infty}$ norm (see (\ref{discrete_w_1_infinity_norm}) for the detailed description):
\[
|||v|||_{W_h^{1,\infty}(\Omega)}
=\|\nabla v\|_{L_h^\infty(\Omega)}
+
\max_{F\in\mathcal E_h}
h_F^{-1}\|[v]\|_{L^\infty(F)},
\]
which extends the classical broken $W^{1,\infty}$ seminorm $\|\nabla v\|_{L_h^\infty(\Omega)}$ by incorporating interface jumps of $v$ across mesh interfaces and the trace of $v$ on $\partial\Omega$. Here $\mathcal{E}_h$ denotes the set of all faces of $\mathcal{T}_h$ and $h_F$ denotes the diameter of the face $F$. This norm is the DG analogue of the standard $W^{1,\infty}$-norm. We establish optimal error estimates in this norm for the symmetric IPDG, hybrid mixed DG, and two representative HDG methods for the Poisson equation on convex polyhedral domains, assuming a polynomial degree of $k \ge 1$.

If the domain is not smooth, it is not straightforward to obtain these estimates. Existing Green's function analyses for pointwise estimates
are developed to estimate the numerical solution itself or its broken gradient \cite{c2005, cc2005, cmy2021, g2006, lc2026}, whereas interface jumps require fundamentally different representation formulas. We would like to emphasize that, if the domain is a convex polyhedron, a direct combination of existing $L^\infty$ estimate (\ref{nm-mn:1})  with trace inequalities inevitably introduces logarithmic factors inherited from (\ref{nm-mn:1}).
More precisely, we will have
\begin{align}
\max_{F\in\mathcal E_h}h_F^{-1}\|[u - u_{h}^{{\rm nm}}]\|_{L^\infty(F)} \leq C \vert \ln h \vert^{\beta}
\left( h^{-1} \|u-\chi\|_{L^{\infty}(\Omega)}+\|\nabla(u-\chi)\|_{L_h^{\infty}(\Omega)}\right).
\label{nm-mn:3}
\end{align}
Here the constant $\beta$ in the above inequality is identical to that in (\ref{nm-mn:1}).
To remove the logarithmic factor in (\ref{nm-mn:3}), one might attempt to exploit the weak discrete maximum principle that has been established for FEM in polygons \cite{s1980} and convex polyhedra \cite{ll2020} with polynomial degrees greater than one. 
The weak discrete maximum principle of DG methods means 
\begin{align}
\label{wmp_dg}
\Vert u_{h}^{{\rm nm}} \Vert_{L^{\infty}(\Omega)} \leq C \Vert u \Vert_{L^{\infty}(\partial\Omega)},
\end{align}
if $u_{h}^{{\rm nm}}$ is discrete harmonic. However, with respect to DG methods, weak
discrete maximum principle (\ref{wmp_dg}) has not been proven. We notice that  
the weak discrete maximum principle of IPDG in \cite{ChibaSaito2019} is different from 
(\ref{wmp_dg}), because it replaces $\Vert u\Vert_{L^{\infty}(\partial\Omega)}$ by 
$\Vert u_{h}^{{\rm nm}}\Vert_{L^{\infty}(\partial\Omega)}$ that is not known beforehand.
Furthermore, a DG analogue of the analysis in \cite{ll2020} requires optimal error estimate of 
$\Sigma_{F\in \mathcal{E}_{h}}\Vert [u - u_{h}^{{\rm nm}}]\Vert_{L^{p}(F)}^{p}$ for some $p$ 
in a neighborhood of $2$, which appears to require an estimate of $\max_{F \in \mathcal{E}_{h}}h_{F}^{-1} 
\Vert [u-u_h^{{\rm nm}}] \Vert_{L^{\infty}(F)}$.
Thus, on convex polyhedral domains, it is not clear how to remove the logarithmic factor 
in (\ref{nm-mn:3}) for DG methods.

We would like to emphasize that estimate (\ref{nm-mn:3}) without the logarithmic factor
is the first step toward obtaining discrete
$W^{2,p}$ stability of DG-type methods for second order linear elliptic PDEs in non-divergence form
with uniformly continuous matrix-valued coefficients. If (\ref{nm-mn:3}) contains the logarithmic factor, the
freezing coefficient technique in \cite{fns2018} will not succeed such that the discrete  $W^{2,p}$ stability
of DG-type methods will be valid only for the case $p = 2$. In \cite{fns2018}, estimate (\ref{nm-mn:3})
does not contain the logarithmic factor because the domain is assumed to be smooth.
Therefore, in order to generalize the results in \cite{fns2018} on convex polyhedral domains, estimate
(\ref{nm-mn:3}) without the logarithmic factor is essentially needed.

To obtain optimal error estimates in the complete discrete $W^{1,\infty}$ norm for DG methods on convex polyhedral 
domains, we develop an interface Green's function framework. Existing Green's function analyses estimate the value 
or gradient of numerical solutions through Green's functions generated by a single discrete delta function. The 
present work differs from this approach by introducing Green's functions generated by the difference between 
discrete delta functions on adjacent elements. These Green's functions naturally represent the interface jumps 
of the DG solutions. So they are called interface Green's functions. Moreover, this design induces a cancellation 
in the singular source term, which leads to a sharper local energy estimate for the interface Green's functions 
and yields an additional factor of $h$ (see (\ref{ee-rG23:1}) in Lemma \ref{ee-rG23}). This new local energy 
estimate allows us to control interface jumps without introducing the logarithmic factors that arise from 
a direct application of existing maximum-norm error estimates.

Although we focus on several representative DG methods, the interface Green's function framework is considerably more general. The construction of interface Green's functions depends only on the local discontinuous structure of the approximation rather than on the particular formulation of the numerical method, or on simplicial meshes. We therefore expect that the proposed methodology can be extended to many other nonconforming methods on general polyhedral meshes.

To illustrate this generality, we further consider the $C^0$ interior penalty method for the biharmonic equation in convex polygonal domains. Existing analyses establish pointwise estimates for the broken Hessian \cite{l2021,hltwz2026}. But the jumps of the normal derivative have remained untreated. By adapting the interface Green's function framework developed for the Poisson equation, we derive maximum-norm estimates for these jumps, thereby demonstrating that the proposed analytical technique extends naturally beyond DG methods for the Poisson equation.

The main contributions of this paper may be summarized as follows:

\begin{itemize}
\item We develop a new family of interface Green's functions together with new representation formulas for interface quantities.

\item We establish a new local energy estimate for interface Green's functions that gains an additional factor of $h$, based on a cancellation technique combined with a dyadic decomposition. This estimate differs fundamentally from the classical Green's function estimates used in previous pointwise analyses and constitutes the principal analytical innovation of this work.

\item We obtain optimal error estimates in complete discrete $W^{1,\infty}$ norm for the IPDG, hybrid mixed DG, and two representative HDG methods for the Poisson problem on convex polyhedral domains.

\item We further apply the proposed framework to the $C^0$ interior penalty method for the biharmonic equation on convex polygonal domains, obtaining better maximum-norm estimates for the jumps of the normal derivative than
what existing works \cite{l2021,hltwz2026} can provide.
\end{itemize}

The remainder of the paper is organized as follows. Section \ref{sec2} introduces notation. Section \ref{sec3} establishes complete discrete $W^{1,\infty}$ estimates for several DG methods applied to the Poisson equation on convex polyhedral domains. Section \ref{sec4} develops the interface Green's function framework and provides the proofs for the main results given in Section \ref{sec3}. Section \ref{sec5} extends the analytical framework to the $C^0$ interior penalty approximation of the biharmonic equation on convex polygonal domains. We conclude this paper with some conclusions in Section \ref{sec6}.

\section{Notations}\label{sec2}

Throughout this paper, we assume $\Omega$ is a bounded and convex polyhedral domain in $\mathbb{R}^{d}$ ($d=2,3$).
Let $\mathcal{T}_h$ be a conforming and quasi-uniform mesh of the domain $\Omega$.
The mesh $\mathcal{T}_{h}$ may consist of either simplices or general polyhedral elements.
For a general polyhedral mesh, we assume that for each polyhedral
element $K$, there is a shape-regular sub-triangulation and the number of sub-elements is uniformly
bounded by a positive constant $\mathcal{N}_d$ independent of the mesh size $h$.
The meshes associated with the different numerical methods will be specified when the methods are introduced.
Denote by $\mathcal{E}_h^o$ the set of all interior faces of $\mathcal{T}_h$
and $\mathcal{E}_h^{\partial}$ the set of all boundary faces.
We define $\mathcal{E}_h=\mathcal{E}_h^o\cup\mathcal{E}_h^{\partial}$ and
$\partial\mathcal{T}_h=\{\partial K: K\in\mathcal{T}_h\}$, where $\partial K$ denotes
the boundary of the element $K$.
Elements are taken to be open sets, whereas mesh faces are taken to be closed sets.
For each element $K\in\mathcal{T}_h$ and $F\in\mathcal{E}_h$, let $h_K$
and $h_F$ be the diameters of $K$ and $F$, respectively. Moreover, we set $h=\max_{K\in\mathcal{T}_h}h_K$.

For any nonnegative integer $s$, an open subset $D\subset\Omega$ and $1\leq p\leq \infty$, let $W^{s,p}(D)$ be the Sobolev spaces \cite{a1975} with
norms
\begin{align*}
\|v\|_{W^{s,p}(D)}^p=\sum_{i=0}^s|v|_{W^{i,p}(D)}^p\quad 1\leq p<\infty,\quad \|v\|_{W^{s,\infty}(D)}=\max_{0\leq i\leq s}|v|_{W^{i,\infty}(D)},
\end{align*}
and seminorms
\begin{align*}
|v|^p_{W^{i,p}(D)}=\sum_{|\alpha|=i}\int_D\left|\frac{\partial^{\alpha}v}{\partial x^{\alpha}}\right|^p dx\quad 1\leq p<\infty,\quad
|v|_{W^{i,\infty}(D)}=\max_{|\alpha|=i}\left\|\frac{\partial^{\alpha}v}{\partial x^{\alpha}}\right\|_{L^{\infty}(D)},
\end{align*}
where
\begin{align*}
\frac{\partial^{\alpha}v}{\partial x^{\alpha}}=\frac{\partial^{|\alpha|}v}{\partial x_1^{\alpha_1}\cdots\partial x_d^{\alpha_d}}
\end{align*}
for the multi-index $\alpha=(\alpha_1,\cdots,\alpha_d),~\alpha_1\geq 0,~\cdots,~\alpha_d\geq 0$, and $|\alpha|=\alpha_1+\cdots+\alpha_d$.
Denote by $W_0^{s,p}(D)~(1<p<\infty)$ the completion of $\mathcal{C}_0^{\infty}(D)$ according to the norm $\|\cdot\|_{W^{s,p}(D)}$, where
$\mathcal{C}_0^{\infty}(D)$ is the space of functions with continuous derivatives of arbitrary order and compact support in $D$. If $p=2$, the
Sobolev spaces $W^{s,2}(D)$ are denoted by $H^s(D)$ with norms $\|\cdot\|_{H^s(D)}$ and seminorms $|\cdot|_{H^s(D)}$. If $s=0$, we set $L^p(D)=W^{0,p}(D)$
with norm $\|\cdot\|_{L^p(D)}$.

Next, we introduce mesh-dependent Sobolev spaces:
\begin{align*}
W^{s,p}_h(D)=\{v: v\in W^{s,p}(K\cap D), ~\forall K\in\mathcal{T}_h\},
\end{align*}
with norms
\begin{align*}
\|v\|_{W_h^{s,p}(D)}^p=\sum_{i=0}^s|v|_{W^{i,p}_h(D)}^p\quad 1\leq p<\infty,\quad \|v\|_{W_h^{s,\infty}(D)}=\max_{0\leq i\leq s}|v|_{W_h^{i,\infty}(D)},
\end{align*}
and seminorms
\begin{align*}
|v|_{W_h^{i,p}(D)}^p=\sum_{K\in\mathcal{T}_h}|v|^p_{W^{i,p}(K\cap D)}\quad 1\leq p<\infty,\quad |v|_{W^{i,\infty}_h(D)}=\max_{K\in\mathcal{T}_h}|v|_{W^{i,\infty}(K\cap D)}.
\end{align*}
If $p=2$, the mesh-dependent Sobolev spaces $W_h^{s,2}(D)$ are denoted by $H_h^s(D)$ with norms $\|\cdot\|_{H_h^s(D)}$ and  seminorms $|\cdot|_{H_h^s(D)}$.
If $s=0$, the mesh-dependent Sobolev spaces $W_h^{0,p}(D)$ are abbreviated by $L_h^p(D)$ with norms $\|\cdot\|_{L_h^p(D)}$.

Let $v \in W^{1,\infty}_{h}(\Omega)$. For each interior face
$F\in\mathcal{E}_h^o$, $\{\nabla v\cdot{\bm n}\}$ and $[v]$ denote the average and jump defined by
\begin{align}
\label{dg_notations}
\{\nabla v\cdot{\bm n}\}=\frac{1}{2}\left(\nabla v^+\cdot{\bm n}^+-\nabla v^-\cdot{\bm n}^-\right),\quad [v]=v^{+}-v^{-},
\end{align}
where $v^+$ and $v^-$ are the restrictions of $v$ on the elements $K^+$ and $K^-$ that share the common face $F$, and
${\bm n}^+$ and ${\bm n}^-$ are the outward unit normal vectors to $\partial K^+$ and $\partial K^-$. For each boundary
face $F\in\mathcal{E}_h^{\partial}$, the average and jump are defined as
$\{\nabla v\cdot{\bm n}\}=\nabla v\cdot{\bm n}$ and $[v]=v$.
In addition, we define the following discrete $W^{1,\infty}$-norm:
\begin{align}
\label{discrete_w_1_infinity_norm}
\vertiii{v}_{W_{h}^{1,\infty}(\Omega)} = \|\nabla v\|_{L_h^{\infty}(\Omega)}
+ \max_{F \in \mathcal{E}_{h}}h_{F}^{-1} \Vert [v] \Vert_{L^{\infty}(F)}.
\end{align}

\section{Complete Discrete $W^{1,\infty}$ Analysis for the Poisson problem}\label{sec3}
In this section, we consider the Poisson equation:
\begin{align}
\label{model}
-\Delta u=f,\quad \text{in}~\Omega; \quad \quad u=0,\quad \text{on}~\partial\Omega.
\end{align}
Here $\Omega$ is a bounded and convex polyhedral domain in $\mathbb{R}^{d}$ ($d=2,3$),
$u$ is the unknown and $f$ is the source term.
For any $f\in L^{2}(\Omega)$, the above problem admits a unique weak solution $u\in
H_0^1(\Omega)\cap H^2(\Omega)$.
By setting ${\bm q}=-\nabla u$, we arrive at the mixed formulation:
\begin{align}
\label{mixed}
{\bm q}+\nabla u=0\quad {\rm in}~\Omega,\quad \nabla\cdot{\bm q}=f\quad {\rm in}~\Omega,\quad
u=0\quad {\rm on}~\partial\Omega.
\end{align}

\subsection{DG methods for the Poisson equation}\label{sec3:1}
For the Poisson equation (\ref{model}), we consider several DG methods: symmetric
IPDG method, hybrid mixed DG method and two representative HDG methods. The hybrid mixed DG method provides the same numerical solution as the standard mixed method which approximates ${\bm q}$ by $H(\text{div})$-conforming Raviart-Thomas elements.

\subsubsection{Interior penalty discontinuous Galerkin methods}

Based on the general polyhedral mesh $\mathcal{T}_h$, we define the following discrete space
\begin{align*}
V_h^{{\rm DG}}=\{v_h\in L^2(\Omega): v_h|_K\in\mathcal{P}^k(K),\quad \forall K\in\mathcal{T}_h\},
\quad (k \geq 1),
\end{align*}
where $\mathcal{P}^k(D)$ denotes the set of polynomials with degree less than or equal to $k$ on the domain $D$.
Then, the IPDG discretization of (\ref{model}) reads as follows: Find $u_h^{\rm DG}\in V_h^{{\rm DG}}$ such that
\begin{align}
\label{IPDG}
a_h(u_h^{{\rm DG}},v_h)=(f,v_h)_{\mathcal{T}_h},\quad \forall v_h\in V_h^{{\rm DG}},
\end{align}
where the bilinear form $a_h$ is defined as:
\begin{align*}
a_h(u_h^{{\rm DG}},v_h)
= & (\nabla u_h^{{\rm DG}},\nabla v_h)_{\mathcal{T}_h}
-\sum_{F\in\mathcal{E}_h}\int_F \left( \{\nabla u_h^{{\rm DG}}\cdot{\bm n}\}[v_h]
+\{\nabla v_h\cdot{\bm n}\}[u_h^{{\rm DG}}] \right) ds \\
& \qquad + \sum_{F\in\mathcal{E}_h} \int_{F} \frac{\lambda}{h_F}[u_h^{{\rm DG}}][v_h] ds,
\end{align*}
where $(\cdot,\cdot)_{\mathfrak{D}}=\sum_{K\in\mathfrak{D}}(\cdot,\cdot)_K$ for any $\mathfrak{D}\subset\mathcal{T}_h$, $(\cdot,\cdot)_K$ denotes the inner product in $L^2(K)$, $\lambda$ is the positive stabilization parameter which is independent of the mesh size $h$ and usually chosen to be
large enough to ensure the stability of (\ref{IPDG}),
$\{\nabla v_h\cdot{\bm n}\}$ and $[v_h]$ are defined as (\ref{dg_notations}).

\subsubsection{Hybridizable discontinuous Galerkin methods}

Based on the mesh $\mathcal{T}_h$ to be specified below, we define
the following discrete spaces:
\begin{align*}
{\bm W}_h=&\{{\bm w}_h\in (L^2(\Omega))^d: {\bm w}_h|_K\in {\bm W}(K),~\forall K\in\mathcal{T}_h\},\\
V_h=&\{v_h\in L^2(\Omega): v_h|_K\in V(K),~\forall K\in\mathcal{T}_h\},\\
\widehat{V}_h=&\{\widehat{v}_h\in L^2(\mathcal{E}_h): \widehat{v}_h|_F\in\mathcal{P}^k(F)~\forall F\in\mathcal{E}_h,~\widehat{v}_h|_F=0~\forall F\in\mathcal{E}_h^{\partial}\},
\end{align*}
where $k\geq 1$ and the local discrete spaces $W(K)$ and $V(K)$ are defined as follows:
\begin{align*}
({\bm W}(K), V(K)) = \begin{cases}
\left((\mathcal{P}^k(K))^d+{\bm x}\mathcal{P}^k(K), \mathcal{P}^k(K)\right),\quad &{\rm hybrid~mixed~DG},
\quad \text{on simplicial meshes};\\
\left((\mathcal{P}^k(K))^d,\mathcal{P}^k(K)\right),\quad &{\rm HDG}-1,\quad \text{on general polyhedral meshes};\\
\left((\mathcal{P}^k(K))^d,\mathcal{P}^{k+1}(K)\right),\quad &{\rm HDG}-2,\quad \text{on general polyhedral meshes}.
\end{cases}
\end{align*}
Then, the HDG discretizations of (\ref{mixed}) read as follows: Find $({\bm q}_h,u_h,\widehat{u}_h)\in
{\bm W}_h\times V_h\times \widehat{V}_h$ such that
\begin{subequations}\label{HDG}
\begin{align}
({\bm q}_h,{\bm w}_h)-(u_h,\nabla\cdot{\bm w}_h)_{\mathcal{T}_h}+\langle\widehat{u}_h,{\bm w}_h\cdot{\bm n}\rangle_{\partial\mathcal{T}_h}=0,\label{HDG:1}\\
-({\bm q}_h,\nabla v_h)_{\mathcal{T}_h}+\langle\widehat{\bm q}_h\cdot{\bm n},v_h\rangle_{\partial\mathcal{T}_h}=(f,v_h),\label{HDG:2}\\
\langle\widehat{\bm q}_h\cdot{\bm n},\widehat{v}_h\rangle_{\partial\mathcal{T}_h}=0,\label{HDG:3}\\
\widehat{\bm q}_h\cdot{\bm n}={\bm q}_h\cdot{\bm n}+\tau(\mathcal{L}u_h-\widehat{u}_h),\quad {\rm on}~\partial\mathcal{T}_h,\label{HDG:4}
\end{align}
\end{subequations}
for any $({\bm w}_h,v_h,\widehat{v}_h)\in {\bm W}_h\times V_h\times \widehat{V}_h$, where ${\bm n}$ is the outward unit normal vector to $\partial K$, $\langle\cdot,\cdot\rangle_{\partial K}$ and $(\cdot,\cdot)$ denote the inner products in $L^2(\partial K)$ and $L^2(\Omega)$, and $\langle\cdot,\cdot\rangle_{\partial \mathfrak{D}}=\sum_{K\in \mathfrak{D}} \langle\cdot,\cdot\rangle_{\partial K}$
for any $\mathfrak{D}\subset\mathcal{T}_h$. Here, $\tau$ denotes the stabilization parameter. On each $K\in\mathcal{T}_h$, it is chosen as:
\begin{align}
\label{HDG_stabs}
\begin{cases}
\tau=0~{\rm on}~\partial K,\quad & {\rm hybrid~mixed~DG};\\
\tau=h_K^{-1}~{\rm on}~F^{\ast}_K~{\rm and}~\tau=0~{\rm on}~\partial K\backslash F_K^{\ast},\quad & {\rm HDG-1};\\
\tau= h_K^{-1}~{\rm on}~\partial K,\quad &{\rm HDG-2}.
\end{cases}
\end{align}
We would like to emphasize that for HDG-1, $F_{K}^{\ast}$ denotes a fixed face of the element $K$
such that $K$ lies on one side of the hyperplane containing $F_{K}^{\ast}$. Therefore, when it is applied
on  general polyhedral mesh, HDG-1 has the restriction that for any element $K\in \mathcal{T}_{h}$, $K$ lies
on one side of the hyperplane containing one of its faces.
$\mathcal{L}$ is a projection operator defined as:
\begin{align*}
\mathcal{L}=\begin{cases}
\mathcal{I},\quad & {\rm hybrid~mixed~DG},~{\rm HDG-1},\\
\Pi_{\widehat{V}},\quad &{\rm HDG-2},
\end{cases}
\end{align*}
where $\mathcal{I}$ is the identity mapping and $\Pi_{\widehat{V}}$ denotes the standard $L^2$-orthogonal projection onto $\mathcal{P}^k(F)$ on each
$F\in\mathcal{E}_h$. Summing the equations in (\ref{HDG}) together, we obtain an equivalent formulation of (\ref{HDG}): Find $({\bm q}_h,u_h,\widehat{u}_h)\in {\bm W}_h\times V_h\times \widehat{V}_h$ such that
\begin{align}
\mathcal{B}({\bm q}_h,u_h,\widehat{u}_h;{\bm w}_h,v_h,\widehat{v}_h)=-(f,v_h),\quad \forall ({\bm w}_h,v_h,\widehat{v}_h)\in {\bm W}_h\times V_h\times \widehat{V}_h,\label{eHDG}
\end{align}
where
\begin{align*}
\mathcal{B}({\bm q}_h,u_h,\widehat{u}_h;{\bm w}_h,v_h,\widehat{v}_h)=&({\bm q}_h,{\bm w}_h)-(u_h,\nabla\cdot{\bm w}_h)_{\mathcal{T}_h}
+\langle\widehat{u}_h,{\bm w}_h\cdot{\bm n}\rangle_{\partial\mathcal{T}_h}\\
&-(\nabla\cdot{\bm q}_h,v_h)_{\mathcal{T}_h}-\langle\tau(\mathcal{L}u_h-\widehat{u}_h),\mathcal{L}v_h-\widehat{v}_h\rangle_{\partial \mathcal{T}_h}
+\langle{\bm q}_h\cdot{\bm n},\widehat{v}_h\rangle_{\partial\mathcal{T}_h}.
\end{align*}

The hybrid mixed DG method, HDG-1 and HDG-2 are introduced in \cite{es2010,cfq2018,qs2016} respectively.
In fact, the hybrid mixed DG method provides the same numerical solution as the mixed method which approximates
${\bm q}$ by $H(\text{div})$-conforming Raviart-Thomas elements.

\subsection{Main results of DG methods for the Poisson equation}\label{sec3:2}

We assume that there is a $s>0$ such that the solution $u$ of the Poisson equation (\ref{model}) satisfies
\begin{align}
\label{model_assumptions}
u \in W^{1+s,\infty}(\Omega).
\end{align}
According to the embedding theorem \cite[Theorem~$1.4.4.1$]{g1985}, (\ref{model_assumptions}) implies
$u \in C^{1,\alpha}(\overline{\Omega})$ with $\alpha\in (0,1)$ depending on $s$. In the rest of this paper, let $\Pi_{k}$, $\Pi_{W}$
and $\Pi_{V}$ be the standard $L^2$-orthogonal projections onto $V_h^{{\rm DG}}$, ${\bm W}_h$ and $V_h$, respectively.

\subsubsection{$L^{\infty}$-norm estimate of the jump of DG solutions along mesh interfaces}

For IPDG \cite{cc2005}, LDG \cite{c2005, g2006}, and IP-HDG \cite{lc2026} methods, the authors established (\ref{nm-mn:1}) and (\ref{nm-mn:2}) for a nonnegative constant $\beta$.
It is well known that $\|\nabla(u-u_h^{{\rm nm}})\|_{L_h^{\infty}(\Omega)}$ defines a norm for conforming finite element
methods, whereas for nonconforming methods, such as IPDG, LDG, IP-HDG, and HDG, it is merely a seminorm.
Therefore, for nonconforming methods, it is important to
obtain a complete discrete $W^{1,\infty}$ estimate of $u - u_{h}^{{\rm nm}}$, which is given by
\begin{align}
\label{discrete_w_1_infinity}
\vertiii{u-u_h^{{\rm nm}}}_{W_{h}^{1,\infty}(\Omega)}
= \|\nabla (u-u_h^{{\rm nm}} ) \|_{L_h^{\infty}(\Omega)}
+ \max_{F \in \mathcal{E}_{h}}h_{F}^{-1} \Vert [u-u_h^{{\rm nm}}] \Vert_{L^{\infty}(F)}.
\end{align}
Consequently, it is critical to estimate $\max_{F \in \mathcal{E}_{h}}h_{F}^{-1} \Vert [u-u_h^{{\rm nm}}] \Vert_{L^{\infty}(F)}$.
One possible approach is to utilize (\ref{nm-mn:1}) and  (\ref{nm-mn:2}) with a trace inequality yielding
\begin{align}
\label{naive_w_1_infinity_estimate}
\max_{F \in \mathcal{E}_{h}}h_{F}^{-1} \Vert [u-u_h^{{\rm nm}}] \Vert_{L^{\infty}(F)}
\leq C |\ln h|^{\beta} \left( h^{-1}\|u-\chi\|_{L^{\infty}(\Omega)}+\|\nabla(u-\chi)\|_{L_h^{\infty}(\Omega)}\right),
\end{align}
where $\beta$ is the same as that in (\ref{nm-mn:1}). For nonconforming methods, $\beta = 0$ has only been proven for
domains with smooth boundaries and polynomial degrees larger than one (see IPDG \cite{cc2005} and LDG
\cite{c2005,g2006} methods). On convex polyhedral domains, existing work \cite{lc2026} can only provide (\ref{nm-mn:1})
with a positive $\beta$ independent of the mesh size $h$. To show $\beta = 0$ in (\ref{nm-mn:1}), one may want to utilize the
weak discrete maximum principle. However, to the best of our knowledge, the weak discrete maximum principle has only been proven for conforming finite
element methods on convex polygons \cite{s1980} and convex polyhedra \cite{ll2020} with polynomial degrees
greater than one. Furthermore, the analogue of the analysis in \cite{ll2020} for nonconforming methods involves
an estimate of $\Sigma_{F\in \mathcal{E}_{h}}\Vert [u - u_{h}^{{\rm nm}}]\Vert_{L^{p}(F)}^{p}$ for some $p$ in a neighborhood of $2$, which
appears to require an estimate of $\max_{F \in \mathcal{E}_{h}}h_{F}^{-1} \Vert [u-u_h^{{\rm nm}}] \Vert_{L^{\infty}(F)}$.
Thus, on convex polyhedral domains,  it is not clear how to remove the logarithmic factor
in (\ref{naive_w_1_infinity_estimate}) for nonconforming methods.

Our Theorem~\ref{mn-jump} shows that the logarithmic factor can be removed in the estimate (\ref{naive_w_1_infinity_estimate})
for the IPDG method and three HDG methods we considered in this paper. The proof of Theorem~\ref{mn-jump} relies on a new
methodology we develop in Section \ref{sec3:3}. We emphasize that this method is sufficiently general and can be extended
to other nonconforming methods. Next, we give a brief description of the new methodology developed herein. In the proof of
the maximum norm error estimate, it is well known that the local energy estimates for the regularized
Green's functions play an important role in determining whether the right-hand sides of the estimates (\ref{nm-mn:1}) and
(\ref{nm-mn:2}) depends on the logarithmic factor, and these estimates are precisely the main source of difficulty in proving Theorem \ref{mn-jump}. To address this issue, using a cancellation technique between neighboring discrete delta functions, we establish a new local energy estimate
of specially designed interface Green's functions (see (\ref{ee-rG23:1}) in Lemma \ref{ee-rG23}), which differs substantially from the corresponding estimates in \cite[Lemma 4.3]{cc2005}, \cite[(4.10)]{dlsw2012}, and \cite[Lemma 3.7]{lc2026}. More precisely, the right-hand side of (\ref{ee-rG23:1}) contains an
additional factor of $h$, which plays an essential role for removing the logarithmic factor in the proof of Theorem \ref{mn-jump} (see Section \ref{sec3:3}).

\begin{theorem}\label{mn-jump}
Let $(u,{\bm q})$ be the solution of problem (\ref{mixed}). Let $u_h^{{\rm DG}}$ and $({\bm q}_h,u_h,\widehat{u}_h)$ denote the
numerical solutions obtained by IPDG and HDG methods defined in (\ref{IPDG}) and (\ref{HDG}), respectively. We assume (\ref{model_assumptions})
holds. Then we have
\begin{align}
\max_{F\in\mathcal{E}_h}h_F^{-1}\|[u-u_h^{{\rm DG}}]\|_{L^{\infty}(F)}\leq C\|\nabla(u-\Pi_ku)\|_{L_h^{\infty}(\Omega)},\label{mn-jump:1}
\end{align}
and
\begin{align}
\max_{F\in\mathcal{E}_h}h_F^{-1}\|[u-u_h]\|_{L^{\infty}(F)}\leq& C\Big(\|\nabla(u-\Pi_Vu)\|_{L_h^{\infty}(\Omega)}
+\|{\bm q}-\Pi_W{\bm q}\|_{L^{\infty}(\Omega)}\Big),\label{mn-jump:2}
\end{align}
where the positive constants $C$ are independent of the mesh size $h$.
\end{theorem}

\subsubsection{Discrete $W^{1,\infty}$-seminorm estimate of DG solutions on elements}

Next, we establish error estimates for $\|\nabla(u-u_h^{{\rm DG}})\|_{L_h^{\infty}(\Omega)}$, $\|\nabla(u-u_h)\|_{L_h^{\infty}(\Omega)}$, and $\|{\bm q}-{\bm q}_h\|_{L^{\infty}(\Omega)}$. We note that the detailed proofs are provided in Section \ref{sec3:4}.

\begin{theorem}\label{mn-g}
Let $(u,{\bm q})$ be the solution of problem (\ref{mixed}). Let $u_h^{{\rm DG}}$ and $({\bm q}_h,u_h,\widehat{u}_h)$ denote the
numerical solutions obtained by IPDG and HDG methods defined in (\ref{IPDG}) and (\ref{HDG}), respectively. We assume (\ref{model_assumptions})
holds. Then we have
\begin{align}
\|\nabla(u-u_h^{{\rm DG}})\|_{L^{\infty}_h(\Omega)}\leq C\|\nabla(u-\Pi_ku)\|_{L_h^{\infty}(\Omega)},
\end{align}
and
\begin{subequations}
\begin{align}
\|\nabla(u-u_h)\|_{L_h^{\infty}(\Omega)}\leq C\left(\|\nabla(u-\Pi_Vu)\|_{L_h^{\infty}(\Omega)}
+\|{\bm q}-\Pi_W{\bm q}\|_{L^{\infty}(\Omega)}\right),\\
\|{\bm q}-{\bm q}_h\|_{L^{\infty}(\Omega)}\leq C\left(\|\nabla(u-\Pi_Vu)\|_{L_h^{\infty}(\Omega)}
+\|{\bm q}-\Pi_W{\bm q}\|_{L^{\infty}(\Omega)}\right),\label{mn-g:3}
\end{align}
\end{subequations}
where the positive constants $C$ are independent of the mesh size $h$.
\end{theorem}

\begin{remark}
\label{remark_complete_discrete_W_1_infinity}
Let $(u,{\bm q})$ be the solution of problem (\ref{mixed}). Let $u_h^{{\rm DG}}$ and $({\bm q}_h,u_h,\widehat{u}_h)$ denote the
numerical solutions obtained by IPDG and HDG methods defined in (\ref{IPDG}) and (\ref{HDG}), respectively. We assume (\ref{model_assumptions})
holds. According to Theorem \ref{mn-jump} and Theorem \ref{mn-g}, we immediately have
\begin{align*}
\vertiii{u - u_{h}^{{\rm DG}}}_{W_{h}^{1,\infty}(\Omega)} \leq & C\|\nabla(u-\Pi_ku)\|_{L_h^{\infty}(\Omega)},\\
\vertiii{u - u_{h}}_{W_{h}^{1,\infty}(\Omega)} \leq & C\left(\|\nabla(u-\Pi_Vu)\|_{L_h^{\infty}(\Omega)}
+\|{\bm q}-\Pi_W{\bm q}\|_{L^{\infty}(\Omega)}\right).
\end{align*}
\end{remark}
\begin{remark}
For the HDG-1 method, if we let $\Pi_V$ be a projection operator from $H^1(\Omega)$ onto $V_h$ such that for each element $K\in\mathcal{T}_h$,
\begin{subequations}\label{proj_HDG-1}
\begin{align}
(\Pi_Vu,v_h)_K=(u,v_h)_K,\quad \forall v_h\in\mathcal{P}^{k-1}(K),\\
\langle\Pi_Vu,\widehat{v}_h\rangle_{F_K^{\ast}}=\langle u,\widehat{v}_h\rangle_{F_K^{\ast}},\quad \forall\widehat{v}_h\in\mathcal{P}^k(F_K^{\ast}),
\end{align}
\end{subequations}
then for the HDG-1 method, the result (\ref{mn-g:3}) can be improved to
\begin{align}
\|{\bm q}-{\bm q}_h\|_{L^{\infty}(\Omega)}\leq C\|{\bm q}-\Pi_W{\bm q}\|_{L^{\infty}(\Omega)},\label{mn-g:4}
\end{align}
where $F_{K}^{\ast}$ is introduced in (\ref{HDG_stabs}) for HDG-1.
In fact, the projection defined in \cite[($4.8$)]{cfq2018} is the vectorization of $\Pi_{V}$ defined in (\ref{proj_HDG-1}).
Moreover, we refer to \cite[Proposition~$4.3$]{cfq2018} for the approximation properties of the projection operator $\Pi_{V}$.

Next, we briefly explain how to obtain (\ref{mn-g:4}). For $x_0\in K_{x_0}$ with $K_{x_0}\in\mathcal{T}_h$, let $(\Gamma,\boldsymbol{\mathcal{R}}_2)$
and $(\boldsymbol{\mathcal{R}}_{2,h},\Gamma_{2,h},\widehat{\Gamma}_{2,h})$ be the solutions of problems (\ref{rG1-2}) and (\ref{rG1-2-HDG}), respectively.
Then, similar to (\ref{rG1-2:1}), from (\ref{proj_HDG-1}) and (\ref{ddd:2}) we have
\begin{align}
-((\Pi_W{\bm q}-{\bm q}_h)\cdot{\bm e}_i)(x_0)=&(\boldsymbol{\mathcal{R}}_{2,h}-\boldsymbol{\mathcal{R}}_2,\Pi_W{\bm q}-{\bm q})+(\nabla(\Gamma_{2,h}-\Gamma),
\Pi_W{\bm q}-{\bm q})_{\mathcal{T}_h}\nonumber\\
&-\langle\Gamma_{2,h}-\widehat{\Gamma}_{2,h},(\Pi_W{\bm q}-{\bm q})\cdot{\bm n}\rangle_{\partial\mathcal{T}_h}-({\bm e}_i\delta_{h,x_0},\Pi_W{\bm q}-{\bm q})
_{\mathcal{T}_h}\label{mn-g:4-proof}\\
\leq&C\|{\bm q}-\Pi_W{\bm q}\|_{L^{\infty}(\Omega)}\Big(\|\boldsymbol{\mathcal{R}}_{2,h}-\boldsymbol{\mathcal{R}}_2\|_{L^1(\Omega)}+
\|\nabla(\Gamma_{2,h}-\Gamma)\|_{L_h^1(\Omega)}\nonumber\\
&+\sum_{K\in\mathcal{T}_h}\|\Gamma_{2,h}-\widehat{\Gamma}_{2,h}\|_{L^1(\partial K)}+1\Big),\nonumber
\end{align}
where ${\bm e}_i$ denotes the $i$-the standard basis vector in $\mathbb{R}^d$. Using the argument presented in the proof of Lemma \ref{rG23-L1} (with Lemma \ref{ee-rG23} replaced by Lemma \ref{rG1-lee}), we can obtain (with the proof omitted)
\begin{align*}
\|\boldsymbol{\mathcal{R}}_{2,h}-\boldsymbol{\mathcal{R}}_2\|_{L^1(\Omega)}+
\|\nabla(\Gamma_{2,h}-\Gamma)\|_{L_h^1(\Omega)}+\sum_{K\in\mathcal{T}_h}\|\Gamma_{2,h}-\widehat{\Gamma}_{2,h}\|_{L^1(\partial K)}\leq C,
\end{align*}
which, together with (\ref{mn-g:4-proof}) and the triangle inequality, yields (\ref{mn-g:4}).
\end{remark}

\subsubsection{$L^{\infty}$-norm estimate of $u_{h} - \widehat{u}_{h}$ of HDG methods along the boundary of elements}

Based on Theorem \ref{mn-jump} and Theorem \ref{mn-g}, we prove the following two corollaries which are specialized for HDG methods. We note that the proofs
of Corollary \ref{mn-HDG} and Corollary \ref{mn-HDG-1} are given in Section \ref{sec3:5}.
\begin{corollary}\label{mn-HDG}
We assume (\ref{model_assumptions}) holds.
Let $(u,{\bm q})$ and $({\bm q}_h,u_h,\widehat{u}_h)$ be the solutions of problems (\ref{mixed}) and (\ref{HDG}), then for each $K\in\mathcal{T}_h$ we have
\begin{align}
\|\tau(\mathcal{L}u_h-\widehat{u}_h)\|_{L^{\infty}(\partial K)}\leq& C\Big(\|{\bm q}-\Pi_W{\bm q}\|_{L^{\infty}(\Omega)}
+\|\nabla(u-\Pi_Vu)\|_{L_h^{\infty}(\Omega)}\Big),\label{mn-HDG:1}
\end{align}
where the positive constant $C$ is independent of the mesh size $h$.
\end{corollary}

\begin{corollary}\label{mn-HDG-1}
We assume (\ref{model_assumptions}) holds.
Let $(u,{\bm q})$ and $({\bm q}_h,u_h,\widehat{u}_h)$ be the solutions of problems (\ref{mixed}) and (\ref{HDG}), then for each $K\in\mathcal{T}_h$ we have
\begin{align}
h_K^{-1}\|u_h-\widehat{u}_h\|_{L^{\infty}(\partial K)}\leq & C\Big(\|{\bm q}-\Pi_W{\bm q}\|_{L^{\infty}(\Omega)}
+\|\nabla(u-\Pi_Vu)\|_{L_h^{\infty}(\Omega)}\Big),\label{mn-HDG-1:1}
\end{align}
for the hybrid mixed DG method and the HDG-1 method, and
\begin{align}
h_K^{-1}\|u_h-\widehat{u}_h\|_{L^{\infty}(\partial K)}\leq & C\Big(\|{\bm q}-\Pi_W{\bm q}\|_{L^{\infty}(\Omega)}
+\|\nabla(u-\Pi_Vu)\|_{L_h^{\infty}(\Omega)}\nonumber\\
&+\|\nabla(u-\Pi_ku)\|_{L_h^{\infty}(\Omega)}\Big),\label{mn-HDG-1:2}
\end{align}
for the HDG-2 method, where the positive constants $C$ are independent of the mesh size $h$.
\end{corollary}

\section{Proofs of main results in Section~\ref{sec3}}\label{sec4}

\subsection{Auxiliary results for DG methods for the Poisson equation and general polyhedral meshes}

We introduce some auxiliary results associated with the IPDG method (\ref{IPDG}) and the HDG method (\ref{HDG}),
and prove an auxiliary result for general polyhedral meshes.

\begin{lemma}\label{dfp}
(1) Let $u$ be the solution of problem (\ref{model}), then for any $\widetilde{v}_1,\widetilde{v}_2\in H^2_h(\Omega)$ we have
\begin{align}
a_h(u,v_h)=(f,v_h)_{\mathcal{T}_h},\quad \forall v_h\in V_h^{{\rm DG}}\quad {\rm and}\quad a_h(\widetilde{v}_1,\widetilde{v}_2)=a_h(\widetilde{v}_2,\widetilde{v}_1).
\end{align}

(2) Let $({\bm q},u)$ be the solution of problem (\ref{mixed}), then for any $({\bm w}_1,v_1,\widehat{v}_1), ({\bm w}_2,v_2,\widehat{v}_2)\in H_h^1(\Omega)\times H_h^1(\Omega)\times L^2(\mathcal{E}_h)$ we have
\begin{align}
\mathcal{B}({\bm q},u,u;{\bm w}_h,v_h,\widehat{v}_h)=-(f,v_h),\quad \forall ({\bm w}_h,v_h,\widehat{v}_h)\in {\bm W}_h\times V_h\times \widehat{V}_h,
\end{align}
and
\begin{align}
\mathcal{B}({\bm w}_1,v_1,\widehat{v}_1;{\bm w}_2,v_2,\widehat{v}_2)=\mathcal{B}({\bm w}_2,v_2,\widehat{v}_2;{\bm w}_1,v_1,\widehat{v}_1).
\end{align}
\end{lemma}
\begin{proof}
This lemma can be proved by the definitions of $\mathcal{B}$, $\mathcal{L}$ and $a_h$ and integration by parts.
\end{proof}

\begin{lemma}\label{inequality}
Let $({\bm q}_h,u_h,\widehat{u}_h)\in {\bm W}_h\times V_h\times \widehat{V}_h$ satisfy (\ref{HDG:1}), then we have
\begin{align}
\|\nabla u_h\|^2_{L^2(K)}+h_K^{-1}\|u_h-\widehat{u}_h\|_{L^2(\partial K)}^2\leq C\left(\|{\bm q}_h\|^2_{L^2(K)}+\|\sqrt{\tau}(\mathcal{L}u_h-\widehat{u}_h)\|
^2_{L^2(\partial K)}\right),
\end{align}
for each $K\in\mathcal{T}_h$.
\end{lemma}
\begin{proof}
The result follows from \cite[Theorem 2.3]{cfq2018}, \cite[Lemma 3.2]{qs2016}, and \cite[Lemma 3.1]{es2010} for HDG-1, HDG-2, and hybrid mixed DG methods, respectively.
\end{proof}

\begin{lemma}[Local energy error estimates]\label{lee}
Let $\Omega_0\subset\Omega_1\subset\Omega$ satisfy $\rho={\rm dist}(\partial\Omega_0,\partial\Omega_1\backslash\partial\Omega)\geq rh$ for $r>1$ sufficiently large.

(1) If $u$ and $u_h^{{\rm DG}}$ are the solutions of problems (\ref{model}) and (\ref{IPDG}), then we have
\begin{align}
&\left(\sum_{F\in\mathcal{E}_h}\left(h_F^{-1}\|[u-u_h^{{\rm DG}}]\|^2_{L^2(F\cap\Omega_0)}+h_F\|\{\nabla(u-u_h^{{\rm DG}})\cdot{\bm n}\}\|^2_{L^2(F\cap\Omega_0)}\right)\right)^{1/2}\nonumber\\
&+\|\nabla(u-u_h^{{\rm DG}})\|_{L_h^2(\Omega_0)}\leq C\Big(h^{-1}\|u-\Pi_ku\|_{L^2(\Omega_1)}+\|\nabla(u-\Pi_ku)\|_{L_h^2(\Omega_1)}\\
&\quad\quad\quad+h|u-\Pi_ku|_{H_h^2(\Omega_1)}\Big)+C\rho^{-1}\|u-u_h^{{\rm DG}}\|_{L^2(\Omega_1)},\nonumber
\end{align}
where $\Pi_k$ is defined in the beginning of Subsection \ref{sec3:2}.

(2) If $({\bm q},u)$ and $({\bm q}_h,u_h,\widehat{u}_h)$ are the solutions of problems (\ref{mixed}) and (\ref{HDG}), then we have
\begin{align}
&\|{\bm q}-{\bm q}_h\|_{L^2(\Omega_0)}+\|\nabla(u-u_h)\|_{L_h^2(\Omega_0)}+\left(\sum_{K\in\mathcal{T}_h}\|\sqrt{\tau}(\mathcal{L}u_h-\widehat{u}_h)\|^2_{L^2(\partial K\cap\Omega_0)}\right)^{1/2}\nonumber\\
\leq& Ch\rho^{-1}\left(\|{\bm q}-{\bm q}_h\|_{L^2(\Omega_1)}+\left(\sum_{K\in\mathcal{T}_h}\|\sqrt{\tau}(\mathcal{L}u_h-\widehat{u}_h)\|^2_{L^2(\partial K\cap\Omega_1)}\right)^{1/2}\right)\\
&+C\rho^{-1}\|u-u_h\|_{L^2(\Omega_1)}+C\Big(h^{-1}\|u-\Pi_Vu\|_{L^2(\Omega_1)}+\|\nabla(u-\Pi_Vu)\|_{L_h^2(\Omega_1)}\nonumber\\
&+\|{\bm q}-\Pi_W{\bm q}\|_{L^2(\Omega_1)}+h\|\nabla({\bm q}-\Pi_W{\bm q})\|_{L_h^2(\Omega_1)}\Big),\nonumber
\end{align}
where $\Pi_V$ and $\Pi_W$ are defined in the beginning of Subsection \ref{sec3:2}.
\end{lemma}
\begin{proof}
The estimates in (1) and (2) follow from \cite[Lemma 4.1]{cc2005} and \cite[Lemma 3.4]{clq2026} respectively.
\end{proof}

\begin{lemma}[Energy error estimates]\label{leee}
(1) Let $u$ and $u_h^{{\rm DG}}$ be the solutions of problems (\ref{model}) and (\ref{IPDG}), then if the stabilization parameter $\lambda$ is sufficiently large, we have
 \begin{align}
 &\|\nabla(u-u_h^{{\rm DG}})\|_{L_h^2(\Omega)}+\Big(\sum_{F\in\mathcal{E}_h}h_F^{-1}\|[u-u_h^{{\rm DG}}]\|^2_{L^2(F)}\Big)^{1/2}\nonumber\\
 \leq&C\|\nabla(u-\Pi_ku)\|_{L_h^2(\Omega)}+C\Big(\sum_{F\in\mathcal{E}_h}\Big(h_F^{-1}\|[u-\Pi_ku]\|^2_{L^2(F)}\\
 &+h_F\|\{\nabla(u-\Pi_ku)\cdot{\bm n}\}\|^2_{L^2(F)}\Big)\Big)^{1/2}.\nonumber
 \end{align}

(2) Let $({\bm q},u)$ and $({\bm q}_h,u_h,\widehat{u}_h)$ be the solutions of problems (\ref{mixed}) and (\ref{HDG}), then we have
\begin{align}
&\|{\bm q}-{\bm q}_h\|_{L^2(\Omega)}+\|\nabla(u-u_h)\|_{L^2_h(\Omega)}+\Big(\sum_{K\in\mathcal{T}_h}\|\sqrt{\tau}(\mathcal{L}u_h-\widehat{u}_h)\|^2_{L^2(\partial K)}\Big)^{1/2}\nonumber\\
\leq&C\Big(h^{-1}\|u-\Pi_Vu\|_{L^2(\Omega)}+\|\nabla(u-\Pi_Vu)\|_{L_h^2(\Omega)}\\
&+\|{\bm q}-\Pi_W{\bm q}\|_{L^2(\Omega)}+h\|\nabla({\bm q}-\Pi_W{\bm q})\|
_{L_h^2(\Omega)}\Big).\nonumber
\end{align}
\end{lemma}
\begin{proof}
The results in (1) can be found in \cite[Theorem 4.17]{de2011}. The results in (2) for hybrid mixed DG, HDG-1 and HDG-2 methods can be found in
\cite{cfq2018, es2010, qs2016}, respectively.
\end{proof}
\begin{lemma}\label{curve-assumption}
For any element $K\in\mathcal{T}_h$ and any two distinct points $a,b\in \overline{K}$, there exist an positive integer $\mathcal{N}_d^K$, a sequence of points $\{y_i\}_{i=0}^{\mathcal{N}_d^K}$ with
$y_0=a$ and $y_{\mathcal{N}_d^K}=b$, and a one-to-one mapping $\boldsymbol{\eta}:[0,1]\rightarrow \mathcal{V}$, such that

(1) $\boldsymbol{\eta}\left(0\right)=a$ and $\boldsymbol{\eta}\left(1\right)=b$;

(2) $\boldsymbol{\eta}(t)\in K$, $\forall t\in(0,1)$;

(3) $\boldsymbol{\eta}$ is Lipschitz on $[0,1]$, and $\mathcal{C}^1$ on $[0,1]$ except at the points $t$ such that $\boldsymbol{\eta}(t)\in \{y_0,\cdots,y_{\mathcal{N}_d^K}\}$;

(4) $\left|\frac{d\boldsymbol{\eta}}{dt}\right|\leq Ch$, $\forall t\in [0,1]$ except at the points $t$ such that $\boldsymbol{\eta}(t)\in \{y_0,\cdots,y_{\mathcal{N}_d^K}\}$.

\noindent Here $1\leq\mathcal{N}_d^K\leq \mathcal{N}_d+2$ is an integer independent of the mesh size $h$ and the element $K$, where $\mathcal{N}_d$ is given in the beginning of Section \ref{sec2}. $C$ is a positive constant independent of the mesh size $h$, $K$, and the points $a,b$. $\mathcal{V}$ is a curve consisting of the line segments $y_0y_1, y_1y_2,\cdots y_{\mathcal{N}_d^{K-1}}y_{\mathcal{N}_d^K}$.
\end{lemma}
\begin{proof}
For any element $K\in\mathcal{T}_{h}$, there exists a shape-regular and conforming sub-triangulation consisting of non-overlapping simplices $\{ T_{i}\}_{i=1}^{\mathcal{N}_{d}^{\prime}}$ such that $\overline{K} = \bigcup_{i=1}^{\mathcal{N}_{d}^{\prime}} \overline{T}_{i}$, where $1 \leq \mathcal{N}_{d}^{\prime} \leq \mathcal{N}_{d}$.
Since $K$ is connected, for any given points $a,b\in \overline{K}$, we can extract a connected chain of adjacent simplices from this sub-triangulation that connects $a$ to $b$. Without loss of generality, we re-index this chain as $T_1, T_2, \cdots, T_m$ with $1 \leq m \leq \mathcal{N}_{d}^{\prime}$, such that $a\in \overline{T}_{1}$, $b\in \overline{T}_{m}$, and any two consecutive simplices $T_i$ and $T_{i+1}$ share a common face $F_{i}= \partial T_{i} \cap \partial T_{i+1}$ for $1\leq i \leq m - 1$.

If $m=1$, then $a, b \in \overline{T}_1$. We set $\mathcal{N}_d^K = 2$, $y_0 = a$, $y_1=c_0$, and $y_2 = b$, where $c_0$ denotes the center of $T_1$. Let $\mathcal{V}$ be the curve consisting of the line segments $y_0y_1$ and $y_1y_2$. Since $T_1$ is convex, $\mathcal{V}$ lies entirely within $\overline{T}_1 \subset \overline{K}$.

If $m>1$, let $c_{i}$ be the barycenter of the common face $F_{i}$ for $1\leq i \leq m - 1$. We set $\mathcal{N}_d^K = m+2$, $y_0 = a$, $y_{i+1} = c_i$ for $0 \leq i \leq m$, and $y_{m+2} = b$, where $c_0$ and $c_m$ denote the centers of $T_1$ and $T_{m}$. Let $\mathcal{V}$ be the curve connecting $y_0, y_1, \cdots, y_{m+2}$ sequentially.
Because each simplex is convex, the segments $y_0y_1, y_1y_2$ lie in $\overline{T}_1$, the segments $y_{i-1}y_i$ ($3 < i < m$) lie in $\overline{T}_{i-1}$, and the segments $y_{m}y_{m+1}, y_{m+1}y_{m+2}$ lie in $\overline{T}_m$. Thus, we have $\mathcal{V} \subset \overline{K}$.

In both cases, the curve $\mathcal{V}$ consists of $\mathcal{N}_d^K \leq m+2 \leq \mathcal{N}_d^{\prime}+2 \leq \mathcal{N}_d+2$ line segments.
Since the diameter of each sub-simplex $T_i$ is bounded by the local mesh size $h$, the length of each line segment is at most $h$.
Consequently, the total arc-length $L$ of the curve $\mathcal{V}$ is bounded by $(m+2) h \leq (\mathcal{N}_d+2) h$.

By using a uniform parametrization with respect to the arc-length, we define a one-to-one mapping $\boldsymbol{\eta}: [0,1] \rightarrow \mathcal{V}$ such that the curve is traversed from $a$ to $b$.
Obviously, $\boldsymbol{\eta}$ satisfies conditions (1), (2) and (3).
Furthermore, the magnitude of the derivative satisfies $\left|\frac{d\boldsymbol{\eta}}{dt}\right| \leq (\mathcal{N}_d+2) h$ at all points where it is differentiable (i.e., except at the values $t$ mapping to the vertices $y_i$). Thus, condition (4) holds with $C = (\mathcal{N}_d+2)$.
This completes the proof.
\end{proof}

\subsection{The Interface Green's Function Framework: Proof of Theorem \ref{mn-jump}}\label{sec3:3}

To illustrate the interface Green's function framework, we provide the complete proof of Theorem \ref{mn-jump}.
Firstly, we review estimates of the Green's function of the Poisson equation on convex polyhedral domains in
$\mathbb{R}^{d}$ ($d=2,3$) in section~\ref{sec3:3.1}. Secondly, to estimate the jump of numerical solutions on
interior mesh interfaces, we introduce
an interface Green's function (see (\ref{rG2})) in section~\ref{sec3:3.2} and represent the jump of numerical solutions
in terms of this interface Green's function and its DG approximations in section~\ref{sec3:3.3}.
Similarly, we introduce the other interface Green's function (see (\ref{rG3})) in section~\ref{sec3:3.4}, and use it and
its DG approximations to represent the numerical solutions on boundary faces in section~\ref{sec3:3.5}.
Then, in section~\ref{sec3:3.6}, we provide a crucial Lemma~\ref{ee-rG23} for The Key Estimate for interface Green's functions
on dyadic decomposition of the domain $\Omega$. Finally, we complete the proof in section~\ref{sec3:3.7}.

\subsubsection{Green's function of the Poisson equation}
\label{sec3:3.1}

For $z\in\Omega$, let $\delta_z$ be the Dirac delta function such that $(\delta_z(x), v(x))=v(z)$ for any $v(x)\in\mathcal{C}(\Omega)$. By virtue of $\delta_z$, we define the Green's function as
\begin{align}
-\Delta G(z,x)=\delta_z\quad {\rm in}~\Omega,\quad G(z,x)=0\quad {\rm on}~\partial\Omega.\label{G}
\end{align}
From \cite[Theorem 2]{c1985}, we know that for each $z\in\Omega$, the above problem has a unique weak solution $G(z,x)\in W_0^{1,q}(\Omega)$ for $1\leq q<\frac{d}{d-1}$. Moreover, from \cite[Theorem 5.1.1 and Theorem 5.1.8]{mr2010}, \cite[Lemma 3.2]{dlsw2012}, \cite[Theorem 1]{glrs2009}, and
\cite[Proposition 4.1]{drs2024} we have
\begin{align}\label{pe-G:0}
|G(z,x)|\leq \begin{cases}
C|x-z|^{-1},\quad &d=3,\\
C\ln\frac{1}{|x-z|},\quad &d=2,
\end{cases}
\end{align}
\begin{align}
\left|\frac{\partial^{\alpha+\beta}G(z,x)}{\partial z^{\alpha}\partial x^{\beta}}\right|\leq C|x-z|^{2-d-|\alpha|-|\beta|},\quad 0<|\alpha+\beta|\leq 2,~
0\leq |\alpha|,|\beta|\leq 1,\label{pe-G}
\end{align}
for $z\neq x$, and
\begin{align}
\frac{|\partial_{z_i}G(z,x)-\partial_{y_i}G(y,x)|}{|z-y|^{\gamma}}\leq C\left(|x-z|^{-d+1-\gamma}+|x-y|^{-d+1-\gamma}\right),\label{pe-G:1}
\end{align}
for all $x,y,z\in \Omega$ and $y\neq z$, where $\gamma\in(0,1)$ is a constant depending on the domain $\Omega$ and $\partial_{z_i}$ denotes the partial derivative with respect to the $i$-th component of $z$.

\subsubsection{Interface Green's function for the jump of numerical solutions on interior mesh interfaces}
\label{sec3:3.2}

In addition to the Dirac delta function $\delta_z$, for each $z\in \overline{K}_z$ with $K_z\in\mathcal{T}_h$, we also define a smooth discrete delta function $\delta_{h,z}$ supported in $K_z$ (cf. \cite[Appendix A]{sw1995}), which satisfies
\begin{align}
\chi(z)=\int_{K_z}\delta_{h,z}(x)\chi(x)dx\quad \forall \chi\in \mathcal{P}^s(K_z),\quad \int_{K_z}\delta_{h,z}(x)dx=1,\label{ddd:1}
\end{align}
and
\begin{align}
|\delta_{h,z}|_{W^{l,p}(\Omega)}\leq Ch^{-l-d\left(1-\frac{1}{p}\right)}\quad \forall 1\leq p\leq \infty,~l=0,1,2,\label{ddd:2}
\end{align}
where $s=k$ for the IPDG, the hybrid mixed DG and the HDG-1 methods, and $s=k+1$ for the HDG-2 method.

Based on $\delta_{h,z}$, we are going to define the interface Green's functions in (\ref{rG2}) and (\ref{rG2-mixed}).
For each $x_{\ast}\in F$ with $F\in\mathcal{E}_h^o$, let $K_{\ast}^+$ and $K_{\ast}^-$
be the elements sharing the common edge/face $F$, then we define the interface Green's function $\Gamma_{{\rm if}}$
(for the jump of numerical solutions along interior mesh interfaces) to be
\begin{align}
-\Delta\Gamma_{{\rm if}}(x_{\ast},x)=\delta_{h,K_{\ast}^+}-\delta_{h,K_{\ast}^-}\quad {\rm in}~\Omega,\quad \Gamma_{{\rm if}}(x_{\ast},x)=0\quad {\rm on}~\partial\Omega,\label{rG2}
\end{align}
where $\delta_{h,K_{\ast}^+}$ and $\delta_{h,K_{\ast}^-}$ denote the discrete delta function $\delta_{h,x_{\ast}}$ supported in $K_{\ast}^+$ and $K_{\ast}^-$
respectively. By setting $\boldsymbol{\mathcal{R}}_{{\rm if}}(x_{\ast},x)=-\nabla\Gamma_{{\rm if}}(x_{\ast},x)$, we arrive at the mixed formulation
\begin{align}
\boldsymbol{\mathcal{R}}_{{\rm if}}+\nabla\Gamma_{{\rm if}}=0\quad {\rm in}~\Omega,\quad
\nabla\cdot\boldsymbol{\mathcal{R}}_{{\rm if}}=\delta_{h,K_{\ast}^+}-\delta_{h,K_{\ast}^-}\quad {\rm in}~\Omega,\quad
\Gamma_{{\rm if}}=0\quad {\rm on}~\partial\Omega.\label{rG2-mixed}
\end{align}

\subsubsection{Representation of the jump of numerical solutions on interior mesh interfaces}
\label{sec3:3.3}

The following representation formulas reduce the estimation of interface jumps to the approximation of
interface Green's functions.
Let $\Gamma_{{\rm if},h}^{{\rm DG}}\in V_h^{{\rm DG}}$ and $(\boldsymbol{\mathcal{R}}_{{\rm if},h},\Gamma_{{\rm if},h},\widehat{\Gamma}_{{\rm if},h})\in {\bm W}_h\times V_h\times \widehat{V}_h$ be the IPDG and HDG approximations of problems (\ref{rG2}) and (\ref{rG2-mixed}) such that
\begin{align}
a_h(\Gamma_{{\rm if},h}^{{\rm DG}},v_h)=(\delta_{h,K_{\ast}^+}-\delta_{h,K_{\ast}^-},v_h),\quad \forall v_h\in V_h^{{\rm DG}},\label{rG2-DG}
\end{align}
and
\begin{align}
\mathcal{B}(\boldsymbol{\mathcal{R}}_{{\rm if},h},\Gamma_{{\rm if},h},\widehat{\Gamma}_{{\rm if},h};{\bm w}_h,v_h,\widehat{v}_h)=-(\delta_{h,K_{\ast}^+}-\delta_{h,K_{\ast}^-},v_h),\quad \forall ({\bm w}_h,v_h,\widehat{v}_h)\in{\bm W}_h\times V_h\times \widehat{V}_h.\label{rG2-HDG}
\end{align}
Then, for the IPDG method, applying Lemma \ref{dfp}, (\ref{ddd:2}), (\ref{rG2}), and (\ref{rG2-DG}), we have
\begin{align}
&\left[\Pi_ku-u_h^{{\rm DG}}\right](x_{\ast})=\left(\delta_{h,K_{\ast}^+}-\delta_{h,K_{\ast}^-},\Pi_ku-u_{h}^{{\rm DG}}\right)\nonumber\\
=&a_h(\Gamma_{{\rm if},h}^{{\rm DG}}-\Gamma_{{\rm if}},\Pi_ku-u)+(\delta_{h,K_{\ast}^+}-\delta_{h,K_{\ast}^-},\Pi_ku-u)\nonumber\\
\leq&C\Big(\sum_{F\in\mathcal{E}_h}\left(h_F\|\{\nabla(\Gamma_{{\rm if}}-\Gamma_{{\rm if},h}^{{\rm DG}})\cdot{\bm n}\}\|_{L^1(F)}+\|[\Gamma_{{\rm if}}-\Gamma_{{\rm if},h}^{{\rm DG}}]\|_{L^1(F)}\right)\nonumber\\
&+\|\nabla(\Gamma_{{\rm if}}-\Gamma_{{\rm if},h}^{{\rm DG}})\|_{L_h^1(\Omega)}\Big)\left(\|\nabla(u-\Pi_ku)\|_{L_h^{\infty}(\Omega)}+h^{-1}\|u-\Pi_ku\|_{L^{\infty}(\Omega)}\right)\label{rG2:1}\\
&+\|\delta_{h,K_{\ast}^+}-\delta_{h,K_{\ast}^-}\|_{L^1(\Omega)}\|u-\Pi_ku\|_{L^{\infty}(\Omega)}\nonumber\\
\leq& C\Big(h+\sum_{F\in\mathcal{E}_h}\left(h_F\|\{\nabla(\Gamma_{{\rm if}}-\Gamma_{{\rm if},h}^{{\rm DG}})\cdot{\bm n}\}\|_{L^1(F)}+\|[\Gamma_{{\rm if}}-\Gamma_{{\rm if},h}^{{\rm DG}}]\|_{L^1(F)}\right)\nonumber\\
&+\|\nabla(\Gamma_{{\rm if}}-\Gamma_{{\rm if},h}^{{\rm DG}})\|_{L_h^1(\Omega)}\Big)\|\nabla(u-\Pi_ku)\|_{L_h^{\infty}(\Omega)}.\nonumber
\end{align}
For the HDG methods, using Lemma \ref{dfp}, (\ref{rG2-mixed}), and (\ref{rG2-HDG}), we have
\begin{align*}
-[\Pi_Vu-u_h](x_{\ast})=&-(\delta_{h,K_{\ast}^+}-\delta_{h,K_{\ast}^-},\Pi_Vu-u_h)_{\mathcal{T}_h}\nonumber\\
=&\mathcal{B}(\boldsymbol{\mathcal{R}}_{{\rm if},h}-\boldsymbol{\mathcal{R}}_{{\rm if}},\Gamma_{{\rm if},h}-\Gamma_{{\rm if}},\widehat{\Gamma}_{{\rm if},h}-\Gamma_{{\rm if}};\Pi_W{\bm q}-{\bm q},\Pi_Vu-u,\Pi_{\widehat{V}}u-u)\nonumber\\
&+\mathcal{B}(\boldsymbol{\mathcal{R}}_{{\rm if}},\Gamma_{{\rm if}},\Gamma_{{\rm if}};\Pi_W{\bm q}-{\bm q},\Pi_Vu-u,\Pi_{\widehat{V}}u-u)=J_1.\nonumber
\end{align*}
For $J_1$, by (\ref{ddd:2}) and the definitions of $\mathcal{B}$ and $\mathcal{L}$, we obtain
\begin{align*}
J_1=&(\boldsymbol{\mathcal{R}}_{{\rm if},h}-\boldsymbol{\mathcal{R}}_{{\rm if}},\Pi_W{\bm q}-{\bm q})+
(\nabla(\Gamma_{{\rm if},h}-\Gamma_{{\rm if}}),\Pi_W{\bm q}-{\bm q})_{\mathcal{T}_h}-\langle\Gamma_{{\rm if},h}-\widehat{\Gamma}_{{\rm if},h},(\Pi_W{\bm q}-{\bm q})\cdot{\bm n}\rangle_{\partial\mathcal{T}_h}\\
&+(\boldsymbol{\mathcal{R}}_{{\rm if},h}-\boldsymbol{\mathcal{R}}_{{\rm if}},\nabla(\Pi_Vu-u))_{\mathcal{T}_h}-\langle
(\boldsymbol{\mathcal{R}}_{{\rm if},h}-\boldsymbol{\mathcal{R}}_{{\rm if}})\cdot{\bm n},\Pi_Vu-u\rangle_{\partial\mathcal{T}_h}\\
&-(\delta_{h,K_{\ast}^+}-\delta_{h,K_{\ast}^-},\Pi_Vu-u)-\langle\tau(\mathcal{L}\Gamma_{{\rm if},h}-\widehat{\Gamma}_{{\rm if},h}),
\mathcal{L}(\Pi_Vu-u)\rangle_{\partial\mathcal{T}_h}\\
\leq&C\Big(\sum_{K\in\mathcal{T}_h}\|\Gamma_{{\rm if},h}-\widehat{\Gamma}_{{\rm if},h}\|_{L^1(\partial K)}+\|\boldsymbol{\mathcal{R}}_{{\rm if}}-\boldsymbol{\mathcal{R}}_{{\rm if},h}\|_{L^1(\Omega)}+h\|\nabla(\boldsymbol{\mathcal{R}}_{{\rm if}}-\boldsymbol{\mathcal{R}}_{{\rm if},h})\|_{L_h^1(\Omega)}\\
&+\|\nabla(\Gamma_{{\rm if}}-\Gamma_{{\rm if},h})\|_{L_h^1(\Omega)}+h\Big)\Big(\|{\bm q}-\Pi_W{\bm q}\|_{L^{\infty}(\Omega)}+\|\nabla(u-\Pi_Vu)\|_{L^{\infty}_h(\Omega)}\Big).
\end{align*}
Therefore,
\begin{align}
-[\Pi_Vu-u_h](x_{\ast})\leq& C\Big(\sum_{K\in\mathcal{T}_h}\|\Gamma_{{\rm if},h}-\widehat{\Gamma}_{{\rm if},h}\|_{L^1(\partial K)}+\|\boldsymbol{\mathcal{R}}_{{\rm if}}-\boldsymbol{\mathcal{R}}_{{\rm if},h}\|_{L^1(\Omega)}+h\|\nabla(\boldsymbol{\mathcal{R}}_{{\rm if}}-\boldsymbol{\mathcal{R}}_{{\rm if},h})\|_{L_h^1(\Omega)}\nonumber\\
&+\|\nabla(\Gamma_{{\rm if}}-\Gamma_{{\rm if},h})\|_{L_h^1(\Omega)}+h\Big)\Big(\|{\bm q}-\Pi_W{\bm q}\|_{L^{\infty}(\Omega)}+\|\nabla(u-\Pi_Vu)\|_{L^{\infty}_h(\Omega)}\Big).\label{rG2:2}
\end{align}

\subsubsection{Interface Green's function for the numerical solutions on boundary faces}
\label{sec3:3.4}

Next, for each $x^{\ast}\in F$ with $F\in\mathcal{E}_h^{\partial}$, let $K^{\ast}$ be the element including the
face $F$, then we define the interface Green's function (for numerical solutions on boundary faces) in (\ref{rG3}) and
its mixed formulation in  (\ref{rG3-mixed}):
\begin{align}
-\Delta\Gamma_{{\rm bf}}(x^{\ast},x)=\delta_{h,x^{\ast}}\quad {\rm in}~\Omega,\quad \Gamma_{{\rm bf}}(x^{\ast},x)=0\quad {\rm on}~\partial\Omega.\label{rG3}
\end{align}
By setting $\boldsymbol{\mathcal{R}}_{{\rm bf}}(x^{\ast},x)=-\nabla\Gamma_{{\rm bf}}(x^{\ast},x)$, we arrive at the mixed formulation
\begin{align}
\boldsymbol{\mathcal{R}}_{{\rm bf}}+\nabla\Gamma_{{\rm bf}}=0\quad {\rm in}~\Omega,\quad
\nabla\cdot\boldsymbol{\mathcal{ R}}_{{\rm bf}}=\delta_{h,x^{\ast}}\quad {\rm in}~\Omega,\quad
\Gamma_{{\rm bf}}=0\quad {\rm on}~\partial\Omega.\label{rG3-mixed}
\end{align}

\subsubsection{Representation of the numerical solutions on boundary faces}
\label{sec3:3.5}

For the primary and mixed formulations (\ref{rG3}) and (\ref{rG3-mixed}), let $\Gamma_{{\rm bf},h}^{{\rm DG}}\in V_h^{{\rm DG}}$ and $(\boldsymbol{\mathcal{R}}_{{\rm bf},h},\Gamma_{{\rm bf},h},\widehat{\Gamma}_{{\rm bf},h})\in {\bm W}_h\times V_h\times\widehat{V}_h$ be their IPDG and HDG
approximations such that
\begin{align}
a_h(\Gamma_{{\rm bf},h}^{{\rm DG}},v_h)=(\delta_{h,x^{\ast}},v_h),\quad \forall v_h\in V_h^{{\rm DG}},\label{rG3-DG}
\end{align}
and
\begin{align}
\mathcal{B}(\boldsymbol{\mathcal{R}}_{{\rm bf},h},\Gamma_{{\rm bf},h},\widehat{\Gamma}_{{\rm bf},h};{\bm w}_h,v_h,\widehat{v}_h)=
-(\delta_{h,x^{\ast}},v_h),\quad \forall ({\bm w}_h,v_h,\widehat{v}_h)\in {\bm W}_h\times V_h\times \widehat{V}_h.\label{rG3-HDG}
\end{align}
Then, by a similar method that has been used for the derivation of (\ref{rG2:1}) and (\ref{rG2:2}), we have
\begin{align}
&[\Pi_ku-u_h^{{\rm DG}}](x^{\ast})=(\delta_{h,x^{\ast}},\Pi_ku-u_h^{{\rm DG}})\nonumber\\
=&a_h(\Gamma_{{\rm bf},h}^{{\rm DG}}-\Gamma_{{\rm bf}},\Pi_ku-u)+(\delta_{h,\ast},\Pi_ku-u)\nonumber\\
\leq&C\Big(h+\sum_{F\in\mathcal{E}_h}\left(h_F\|\{\nabla(\Gamma_{{\rm bf}}-\Gamma_{{\rm bf},h}^{{\rm DG}})\cdot{\bm n}\}\|_{L^1(F)}+\|[\Gamma_{{\rm bf}}-\Gamma_{{\rm bf},h}^{{\rm DG}}]\|_{L^1(F)}\right)\label{rG3:1}\\
&+\|\nabla(\Gamma_{{\rm bf}}-\Gamma_{{\rm bf},h}^{{\rm DG}})\|_{L_h^1(\Omega)}\Big)\|\nabla(u-\Pi_ku)\|_{L_h^{\infty}(\Omega)},\nonumber
\end{align}
and
\begin{align}
&-[\Pi_Vu-u_h](x^{\ast})=-(\delta_{h,x^{\ast}},\Pi_Vu-u_h)\nonumber\\
=&\mathcal{B}(\boldsymbol{\mathcal{R}}_{{\rm bf},h}-\boldsymbol{\mathcal{R}}_{{\rm bf}},\Gamma_{{\rm bf},h}-\Gamma_{{\rm bf}},\widehat{\Gamma}_{{\rm bf},h}
-\Gamma_{{\rm bf}};\Pi_W{\bm q}-{\bm q},\Pi_Vu-u,\Pi_{\widehat{V}}u-u)\nonumber\\
&+\mathcal{B}(\boldsymbol{\mathcal{R}}_{{\rm bf}},\Gamma_{{\rm bf}},\Gamma_{{\rm bf}};\Pi_W{\bm q}-{\bm q},\Pi_Vu-u,\Pi_{\widehat{V}}u-u)\nonumber\\
=&(\boldsymbol{\mathcal{R}}_{{\rm bf},h}-\boldsymbol{\mathcal{R}}_{{\rm bf}},\Pi_W{\bm q}-{\bm q})+
(\nabla(\Gamma_{{\rm bf},h}-\Gamma_{{\rm bf}}),\Pi_W{\bm q}-{\bm q})_{\mathcal{T}_h}\nonumber\\
&-\langle\Gamma_{{\rm bf},h}-\widehat{\Gamma}_{{\rm bf},h},(\Pi_W{\bm q}-{\bm q})\cdot{\bm n}\rangle_{\partial\mathcal{T}_h}
+(\boldsymbol{\mathcal{R}}_{{\rm bf},h}-\boldsymbol{\mathcal{R}}_{{\rm bf}},\nabla(\Pi_Vu-u))_{\mathcal{T}_h}\label{rG3:2}\\
&-\langle(\boldsymbol{\mathcal{R}}_{{\rm bf},h}-\boldsymbol{\mathcal{R}}_{{\rm bf}})\cdot{\bm n},\Pi_Vu-u\rangle_{\partial\mathcal{T}_h}
-\langle\tau(\mathcal{L}\Gamma_{{\rm bf},h}-\widehat{\Gamma}_{{\rm bf},h}),\mathcal{L}(\Pi_Vu-u)\rangle_{\partial\mathcal{T}_h}\nonumber\\
&-(\delta_{h,x^{\ast}},\Pi_Vu-u)\nonumber\\
\leq&C\Big(\sum_{K\in\mathcal{T}_h}\|\Gamma_{{\rm bf},h}-\widehat{\Gamma}_{{\rm bf},h}\|_{L^1(\partial K)}+\|\boldsymbol{\mathcal{R}}_{{\rm bf}}-\boldsymbol{\mathcal{R}}_{{\rm bf},h}\|_{L^1(\Omega)}+\|\nabla(\Gamma_{{\rm bf}}-\Gamma_{{\rm bf},h})\|_{L_h^1(\Omega)}\nonumber\\
&+h\|\nabla(\boldsymbol{\mathcal{R}}_{{\rm bf}}-\boldsymbol{\mathcal{R}}_{{\rm bf},h})\|_{L_h^1(\Omega)}+h\Big)
\Big(\|\Pi_W{\bm q}-{\bm q}\|_{L^{\infty}(\Omega)}+\|\nabla(u-\Pi_Vu)\|_{L_h^{\infty}(\Omega)}\Big).\nonumber
\end{align}

\subsubsection{The Key Estimate for interface Green's functions on dyadic decomposition of the domain}
\label{sec3:3.6}

Furthermore, we introduce a dyadic decomposition of the domain $\Omega$. Without loss of generality, we assume
that ${\rm diam}(\Omega)\leq 1$. Then, for a fixed point $z\in \overline{\Omega}$, we define
\begin{align*}
\Omega_{\ast}(z)=\{x\in\Omega: |x-z|\leq Mh\},\quad \Omega_j(z)=\{x\in\Omega: d_{j+1}\leq|x-z|\leq d_j\},
\end{align*}
where $d_j=2^{-j}$ for $j=0,1,\cdots,I$ with $d_{I+1}\leq Mh<d_I~(I\thickapprox|\ln h|)$. Here the constant
$M \geq 3$ is independent of the mesh size $h$ and will be determined in Section \ref{sec3:3.7}. Therefore,
\begin{align*}
\Omega=\Omega_{\ast}(z)\cup\bigcup_{j=0}^I\Omega_j(z).
\end{align*}
In addition, we define
\begin{align*}
\Omega_j^1(z)=\Omega_{j-1}(z)\cup\Omega_j(z)\cup\Omega_{j+1}(z),\\
\Omega_j^2(z)=\Omega_{j-2}(z)\cup\Omega_j^1(z)\cup\Omega_{j+2}(z),\\
\Omega_j^3(z)=\Omega_{j-3}(z)\cup\Omega_j^2(z)\cup\Omega_{j+3}(z),\\
\Omega_j^4(z)=\Omega_{j-4}(z)\cup\Omega_j^3(z)\cup\Omega_{j+4}(z).
\end{align*}

\begin{lemma} (The Key Estimate for Interface Green's Functions)
\label{ee-rG23}
Let $\Gamma_{{\rm if}}(x_{\ast},x)$ and $\Gamma_{{\rm bf}}(x^{\ast},x)$ be the interface Green's functions defined
in (\ref{rG2}) and (\ref{rG3}) respectively, then we have
\begin{align}
|\Gamma_{{\rm if}}|_{H^2(\Omega_j^2(x_{\ast}))}\leq Chd_j^{-1-\frac{d}{2}},\quad |\Gamma_{{\rm bf}}|_{H^2(\Omega_j^2(x^{\ast}))}\leq Chd_j^{-1-\frac{d}{2}}.\label{ee-rG23:1}
\end{align}
\end{lemma}

\begin{remark}
\label{remark_key_lemma}
This lemma is the key analytical ingredient of the paper. Unlike classical Green-function estimates, its proof exploits cancellation between neighboring discrete delta functions, yielding one additional factor of \(h\).
This additional factor is precisely what removes the logarithmic dependence in the subsequent interface estimates.
\end{remark}

\begin{proof}
Let $\varphi\in\mathcal{C}_0^{\infty}(\Omega_j^3(x_{\ast}))$ such that $\varphi\equiv 1$ in $\Omega_j^2(x_{\ast})$ with
$|\varphi|_{W^{l,\infty}(\Omega)}\leq Cd_j^{-l}$ for $l=0,1,2$. Then we have
\begin{align*}
-\Delta(\varphi\Gamma_{{\rm if}})=-\Gamma_{{\rm if}}\Delta\varphi-2\nabla\varphi\cdot\nabla\Gamma_{{\rm if}},
\end{align*}
and
\begin{align}
|\Gamma_{{\rm if}}|_{H^2(\Omega_j^2(x_{\ast}))}\leq |\varphi\Gamma_{{\rm if}}|_{H^2(\Omega)}\leq C\left(d_j^{-2}\|\Gamma_{{\rm if}}\|_{L^2(\Omega_j^3(x_{\ast}))}+d_j^{-1}\|\nabla\Gamma_{{\rm if}}\|_{L^2(\Omega_j^3(x_{\ast}))}\right).\label{ee-rG23-proof:1}
\end{align}
For any $x\in\Omega_j^3(x_{\ast})$, using the definitions of the Green's function $G$ (see (\ref{G})) and the Dirac delta function $\delta_x$, we obtain
\begin{align}
\Gamma_{{\rm if}}(x_{\ast},x)=&(\delta_x(y),\Gamma_{{\rm if}}(x_{\ast},y))
= (\nabla\Gamma_{{\rm if}}(x_{\ast},y),\nabla G(x,y)) \label{ee-rG23-proof:ad1}\\
=& (\delta_{h,K_{\ast}^+}(y), G(x,y))_{K_{\ast}^+}-(\delta_{h,K_{\ast}^-}(y),G(x,y))_{K_{\ast}^-},
\nonumber
\end{align}
and
\begin{align}
\nabla_x\Gamma_{{\rm if}}(x_{\ast},x)=& (\delta_{h,K_{\ast}^+}(y), \nabla_xG(x,y))_{K_{\ast}^+}-(\delta_{h,K_{\ast}^-}(y),\nabla_xG(x,y))_{K_{\ast}^-} .\label{ee-rG23-proof:ad2}
\end{align}
Here, for any $x\in\Omega_j^3(x_{\ast})$, we denote by
\begin{align*}
(\delta_{h,K_{\ast}^+}(y), \nabla_xG(x,y))_{K_{\ast}^+}
= \left( (\delta_{h,K_{\ast}^+}(y), \partial_{x_{1}} G(x,y))_{K_{\ast}^+}, \cdots,
(\delta_{h,K_{\ast}^+}(y), \partial_{x_{d}} G(x,y))_{K_{\ast}^+} \right)\in \mathbb{R}^{d}.
\end{align*}

According to (\ref{pe-G}) and the fact that $x \in \Omega_{j}^{3}(x_{\ast})$, $G(x,y)$ and $\nabla_xG(x,y)$
are differentiable with respect to the second variable $y\in K_{\ast}^{+} \cup K_{\ast}^{-}$ and satisfy
\begin{align}
\left|\nabla_yG(x,y)\right|\leq C|x-y|^{1-d},\quad \left|\nabla_y\nabla_xG(x,y)\right|\leq C|x-y|^{-d}.\label{Taylor_expansions_ad1}
\end{align}
For any $y \in K_{\ast}^{+}$, we denote by $\boldsymbol{\eta}_{y}^{+}$ a curve satisfying (2), (3), (4) in Lemma \ref{curve-assumption}
with $\boldsymbol{\eta}_{y}^{+}(0) = x_{\ast}$ and $\boldsymbol{\eta}_{y}^{+}(1) = y$. Similarly,
for any $y \in K_{\ast}^{-}$, we denote by $\boldsymbol{\eta}_{y}^{-}$ a curve satisfying (2), (3), (4) in Lemma \ref{curve-assumption}
with $\boldsymbol{\eta}_{y}^{-}(0) = x_{\ast}$ and $\boldsymbol{\eta}_{y}^{-}(1) = y$.
Then,
\begin{subequations}
\label{Taylor_expansions1}
\begin{align}
G(x,y)=G(x,x_{\ast})+ \int_{0}^{1} \nabla_y G(x,\boldsymbol{\eta}_{y}^{+}(s))
\cdot\frac{d \boldsymbol{\eta}_{y}^{+}(s)}{ds}ds,
\quad \forall y\in K_{\ast}^+,\\
G(x,y)=G(x,x_{\ast})+\int_{0}^{1} \nabla_y G(x,\boldsymbol{\eta}_{y}^{-}(s))
\cdot\frac{d \boldsymbol{\eta}_{y}^{-}(s)}{ds}ds,
\quad \forall y\in K_{\ast}^{-},
\end{align}
\end{subequations}
and
\begin{subequations}
\label{Taylor_expansions2}
\begin{align}
\nabla_xG(x,y)=\nabla_xG(x,x_{\ast})+\int_{0}^{1}\nabla_x\left(\nabla_y G(x,\boldsymbol{\eta}_{y}^{+}(s)
\cdot \frac{d \boldsymbol{\eta}_{y}^{+}(s)}{ds} \right) ds,
\quad \forall y\in K_{\ast}^+,\\
\nabla_xG(x,y)=\nabla_xG(x,x_{\ast})+\int_{0}^{1}\nabla_x\left(\nabla_y G(x,\boldsymbol{\eta}_{y}^{-}(s)
\cdot \frac{d \boldsymbol{\eta}_{y}^{-} (s)}{ds} \right) ds,
\quad \forall y\in K_{\ast}^{-}.
\end{align}
\end{subequations}
Moreover, Lemma \ref{curve-assumption} implies that for any $s \in (0,1)$ and any $x \in \Omega_{j}^{3}(x_{\ast})$,
\begin{subequations}
\label{Taylor_points_props}
\begin{align}
\label{Taylor_points_prop1}
& \boldsymbol{\eta}_{y}^{+}(s) \in K_{\ast}^{+} \quad \forall y\in K_{\ast}^{+},\quad \text{ and }\quad
\boldsymbol{\eta}_{y}^{-}(s) \in K_{\ast}^{-} \quad \forall y\in K_{\ast}^{-}, \\
\label{Taylor_points_prop2}
& \left| x - \boldsymbol{\eta}_{y}^{+}(s) \right| \geq Cd_{j} \quad \forall y\in K_{\ast}^{+}, \text{ and }
\left| x - \boldsymbol{\eta}_{y}^{-}(s) \right| \geq Cd_{j} \quad \forall y\in K_{\ast}^{-},
\end{align}
\end{subequations}
where the constant $C$ is independent of the mesh size $h$, $x\in \Omega_{j}^{3}(x_{\ast})$, and
$y \in K_{\ast}^{+}\cup K_{\ast}^{-}$.
By (\ref{ddd:1}), we have that for any $x\in \Omega_{j}^{3}(x_{\ast})$,
\begin{subequations}
\label{Taylor_zero_terms}
\begin{align}
(\delta_{h,K_{\ast}^{+}}(y), G(x,x_{\ast}))_{K_{\ast}^{+}}
= & (\delta_{h,K_{\ast}^{-}}(y), G(x,x_{\ast}))_{K_{\ast}^{-}}, \\
(\delta_{h,K_{\ast}^{+}}(y), \nabla_x G(x,x_{\ast}))_{K_{\ast}^{+}}
= & (\delta_{h,K_{\ast}^{-}}(y), \nabla_x G(x,x_{\ast}))_{K_{\ast}^{-}}.
\end{align}
\end{subequations}
The above identities (\ref{Taylor_expansions1}), (\ref{Taylor_expansions2}), and (\ref{Taylor_zero_terms}), together
with (\ref{ddd:2}), (\ref{ee-rG23-proof:ad1})-(\ref{Taylor_expansions_ad1}), (4) in Lemma \ref{curve-assumption} and
(\ref{Taylor_points_props}), lead to
\begin{align*}
\left|\Gamma_{{\rm if}}(x_{\ast},x)\right|\leq&\int_{K_{\ast}^+}\left|\delta_{h,K_{\ast}^+}\right|
\left(\int_{0}^{1} \left|\nabla_y G(x,\boldsymbol{\eta}_{y}^{+}(s))\right|
\left|\frac{d \boldsymbol{\eta}_{y}^{+}(s)}{ds}\right|ds\right) dy\\
&+\int_{K_{\ast}^-}\left|\delta_{h,K_{\ast}^-}\right|
\left(\int_{0}^{1} \left|\nabla_y G(x,\boldsymbol{\eta}_{y}^{-}(s))\right|
\left|\frac{d \boldsymbol{\eta}_{y}^{-}(s)}{ds}\right|ds\right) dy\\
\leq& C\int_{K_{\ast}^+}h\left|\delta_{h,K_{\ast}^+}\right|
\left(\int_{0}^{1}\left| x - \boldsymbol{\eta}_{y}^{+}(s) \right|^{1-d} ds \right) dy\\
&+C\int_{K_{\ast}^-}h\left|\delta_{h,K_{\ast}^-}\right|
\left(\int_{0}^{1}\left| x - \boldsymbol{\eta}_{y}^{-}(s) \right|^{1-d} ds \right)dy\leq Chd_j^{1-d},
\end{align*}
and
\begin{align*}
\left| \nabla_{x} \Gamma_{{\rm if}}(x_{\ast},x)\right|\leq&\int_{K_{\ast}^+}\left|\delta_{h,K_{\ast}^+}\right|
\left(\int_{0}^{1} \left| \nabla_{x} \nabla_y G(x,\boldsymbol{\eta}_{y}^{+}(s))\right|
\left|\frac{d \boldsymbol{\eta}_{y}^{+}(s)}{ds}\right|ds\right) dy\\
&+\int_{K_{\ast}^-}\left|\delta_{h,K_{\ast}^-}\right|
\left(\int_{0}^{1} \left| \nabla_{x} \nabla_y G(x,\boldsymbol{\eta}_{y}^{-}(s))\right|
\left|\frac{d \boldsymbol{\eta}_{y}^{-}(s)}{ds}\right|ds\right) dy\\
\leq& C\int_{K_{\ast}^+}h\left|\delta_{h,K_{\ast}^+}\right|
\left(\int_{0}^{1}\left| x - \boldsymbol{\eta}_{y}^{+}(s) \right|^{-d} ds \right) dy\\
&+C\int_{K_{\ast}^-}h\left|\delta_{h,K_{\ast}^-}\right|
\left(\int_{0}^{1}\left| x - \boldsymbol{\eta}_{y}^{-}(s) \right|^{-d} ds \right)dy\leq Chd_j^{-d}.
\end{align*}
Therefore, the above two estimates and (\ref{ee-rG23-proof:1}) imply
\begin{align*}
|\Gamma_{{\rm if}}|_{H^2(\Omega_j^2(x_{\ast}))}\leq Chd_j^{-1-\frac{d}{2}}.
\end{align*}

Let $\varphi_1\in\mathcal{C}_0^{\infty}(\Omega_j^3(x^{\ast}))$ satisfy $\varphi_1\equiv1$ in $\Omega_j^2(x^{\ast})$
with $|\varphi_1|_{W^{l,\infty}(\Omega)}\leq Cd_j^{-l}$ for $l=0,1,2$. Then, similar to (\ref{ee-rG23-proof:1}), we have
\begin{align}
|\Gamma_{{\rm bf}}|_{H^2(\Omega_j^2(x^{\ast}))}\leq C\left(d_j^{-2}\|\Gamma_{{\rm bf}}\|_{L^2(\Omega_j^3(x^{\ast}))}+d_j^{-1}\|\nabla\Gamma_{{\rm bf}}\|_{L^2(\Omega_j^3(x^{\ast}))}\right).\label{ee-rG23-proof:2}
\end{align}
For any $x\in\Omega_j^3(x^{\ast})$, according to the definitions of $G$ (see (\ref{G})) and $\delta_x$ we obtain
\begin{align}
\Gamma_{{\rm bf}}(x^{\ast},x)=(\delta_{h,x^{\ast}}(y), G(x,y))_{K^{\ast}},\label{ee-rG23-proof:3}
\end{align}
and
\begin{align}
\nabla_x\Gamma_{{\rm bf}}(x^{\ast},x)=(\delta_{h,x^{\ast}}(y), \nabla_xG(x,y))_{K^{\ast}}.\label{ee-rG23-proof:4}
\end{align}
From (\ref{pe-G}) and the fact that $x \in \Omega_{j}^{3}(x^{\ast})$, $G(x,y)$ and $\nabla_xG(x,y)$ are
differentiable with respect to the second variable $y \in K^{\ast}$. For any $y \in K^{\ast}$, we denote by
$\boldsymbol{\eta}_{y}$ a curve satisfying (2), (3), (4) in Lemma \ref{curve-assumption} with $\boldsymbol{\eta}_{y}(0) = x^{\ast}$
and $\boldsymbol{\eta}_{y}(1) = y$. Then,
\begin{align*}
G(x,y)= G(x, x^{\ast}) +  \int_{0}^{1} \nabla_{y} G(x,\boldsymbol{\eta}_{y}(s))\cdot
\frac{d\boldsymbol{\eta}_{y}(s)}{ds}ds,\quad \forall y\in K^{\ast},
\end{align*}
and
\begin{align*}
\nabla_x G(x,y)= \nabla_{x} G(x,x^{\ast}) + \int_{0}^{1} \nabla_x\left(\nabla_yG(x,\boldsymbol{\eta}_{y}(s))
\cdot \frac{d \boldsymbol{\eta}_{y}(s)}{ds} \right) ds ,\quad \forall y\in K^{\ast}.
\end{align*}
Since $G(x,x^{\ast})=0$ for any $x\in \Omega_{j}^{3}(x^{\ast})$ and $x^{\ast}\in F$ with $F\in \mathcal{E}_h^{\partial}$,
\begin{align*}
\nabla_{x} G(x,x^{\ast}) = \boldsymbol{0}, \quad \forall x\in \Omega_{j}^{3}(x^{\ast}).
\end{align*}
Therefore we have
\begin{align}
\label{boundary_eq1}
G(x,y)= \int_{0}^{1} \nabla_{y} G(x,\boldsymbol{\eta}_{y}(s))\cdot
\frac{d\boldsymbol{\eta}_{y}(s)}{ds}ds,\quad \forall y\in K^{\ast},
\end{align}
and
\begin{align}
\label{boundary_eq2}
\nabla_x G(x,y)= \int_{0}^{1} \nabla_x\left(\nabla_yG(x,\boldsymbol{\eta}_{y}(s))
\cdot \frac{d \boldsymbol{\eta}_{y}(s)}{ds} \right) ds ,\quad \forall y\in K^{\ast}.
\end{align}
Moreover, Lemma \ref{curve-assumption} implies that for any $s \in (0,1)$ and any $x\in \Omega_{j}^{3}(x^{\ast})$,
\begin{align}
\label{distance_x_y_boundary}
\boldsymbol{\eta}_{y}(s) \in K^{\ast} \quad \forall y \in K^{\ast}, \quad \text{ and } \quad
\vert x - \boldsymbol{\eta}_{y}(s) \vert \geq Cd_{j}\quad \forall y \in K^{\ast},
\end{align}
where the constant $C$ is independent of the mesh size $h$, $x^{\ast} \in \Omega_{j}^{3}(x^{\ast})$
and $y \in K^{\ast}$.
The above two identities (\ref{boundary_eq1}) and (\ref{boundary_eq2}), together with (\ref{pe-G}),
(\ref{ddd:2}), (\ref{ee-rG23-proof:3}), (\ref{ee-rG23-proof:4}), (4) in Lemma \ref{curve-assumption} and
(\ref{distance_x_y_boundary}), result in
\begin{align*}
\left|\Gamma_{{\rm bf}}(x^{\ast},x)\right|\leq&\int_{K^{\ast}}\left|\delta_{h,x^{\ast}}\right|
\left( \int_{0}^{1} \left|\nabla_{y} G(x, \boldsymbol{\eta}_{y}(s) ))\right|
\left| \frac{d \boldsymbol{\eta}_{y}(s)}{ds} \right| ds \right) dy\\
\leq& C\int_{K^{\ast}}h\left|\delta_{h,x^{\ast}}\right|
\left(\int_{0}^{1}\left| x - \boldsymbol{\eta}_{y}(s) \right|^{1-d}ds \right) dy\leq Chd_j^{1-d},
\end{align*}
and
\begin{align*}
\left|\nabla_x\Gamma_{{\rm bf}}(x^{\ast},x)\right|\leq&\int_{K^{\ast}}\left|\delta_{h,x^{\ast}}\right|
\left( \int_{0}^{1} \left| \nabla_{x} \nabla_{y} G(x, \boldsymbol{\eta}_{y}(s) ))\right|
\left| \frac{d \boldsymbol{\eta}_{y}(s)}{ds} \right| ds \right) dy\\
\leq& C\int_{K^{\ast}}h\left|\delta_{h,x^{\ast}}\right|
\left(\int_{0}^{1}\left| x - \boldsymbol{\eta}_{y}(s) \right|^{-d}ds \right) dy\leq Chd_j^{-d}.
\end{align*}
Thus, combining (\ref{ee-rG23-proof:2}) and the above two estimates, we get
\begin{align*}
|\Gamma_{{\rm bf}}|_{H^2(\Omega_j^2(x^{\ast}))}\leq Chd_j^{-1-\frac{d}{2}}.
\end{align*}
\end{proof}

\subsubsection{The final step of the proof of Theorem \ref{mn-jump}}
\label{sec3:3.7}

Based on Lemma \ref{ee-rG23}, we prove the following Lemma~\ref{rG23-L1}, which, together with (\ref{rG2:1}), (\ref{rG2:2}), (\ref{rG3:1}), (\ref{rG3:2}) and the fact that $(\Pi_{k}u-u,1)_{K}=(\Pi_{V}u-u,1)_{K}=0$ for any $K \in \mathcal{T}_{h}$, leads to the desired results in Theorem \ref{mn-jump}.
\begin{lemma}\label{rG23-L1}
Let $(\boldsymbol{\mathcal{R}}_{{\rm if}},\Gamma_{{\rm if}})$ and $(\boldsymbol{\mathcal{R}}_{{\rm bf}},\Gamma_{{\rm bf}})$ be the interface Green's functions defined in (\ref{rG2-mixed}) and (\ref{rG3-mixed}). Let $\Gamma_{{\rm if},h}^{{\rm DG}}$, $\Gamma_{{\rm bf},h}^{{\rm DG}}$, $(\boldsymbol{\mathcal{R}}_{{\rm if},h},\Gamma_{{\rm if},h},\widehat{\Gamma}_{{\rm if},h})$ and $(\boldsymbol{\mathcal{R}}_{{\rm bf},h},\Gamma_{{\rm bf},h},\widehat{\Gamma}_{{\rm bf},h})$ be their IPDG and HDG approximations defined in (\ref{rG2-DG}), (\ref{rG3-DG}), (\ref{rG2-HDG}), and (\ref{rG3-HDG}).
Then, for ${\rm s\in\{if, bf\}}$ we have
\begin{align}
\sum_{F\in\mathcal{E}_h}\left(h_F\|\{\nabla(\Gamma_{{\rm s}}-\Gamma_{{\rm s},h}^{{\rm DG}})\cdot{\bm n}\}\|_{L^1(F)}+\|[\Gamma_{{\rm s}}-\Gamma_{{\rm s},h}^{{\rm DG}}]\|_{L^1(F)}\right)+\|\nabla(\Gamma_{{\rm s}}-\Gamma_{{\rm s},h}^{{\rm DG}})\|_{L_h^1(\Omega)}\leq Ch,\label{rG23-L1-DG}
\end{align}
and
\begin{align}
&\sum_{K\in\mathcal{T}_h}\|\Gamma_{{\rm s},h}-\widehat{\Gamma}_{{\rm s},h}\|_{L^1(\partial K)}
+\|\boldsymbol{\mathcal{R}}_{{\rm s}}-\boldsymbol{\mathcal{R}}_{{\rm s},h}\|_{L^1(\Omega)}\nonumber\\
&+\|\nabla(\Gamma_{{\rm s}}-\Gamma_{{\rm s},h})\|_{L_h^1(\Omega)}+h\|\nabla(\boldsymbol{\mathcal{R}}_{{\rm s}}-\boldsymbol{\mathcal{R}}_{{\rm s},h})\|_{L_h^1(\Omega)}\leq Ch.\label{rG23-L1-HDG}
\end{align}
\end{lemma}
\begin{proof}
For the result (\ref{rG23-L1-DG}), we only prove the case of ${\rm s=if}$, since the case of ${\rm s=bf}$ can be proved similarly.
Using (\ref{ddd:2}) and Lemma \ref{leee}, we obtain
\begin{align}
&\sum_{F\in\mathcal{E}_h}\left(h_F\|\{\nabla(\Gamma_{{\rm if}}-\Gamma_{{\rm if},h}^{{\rm DG}})\cdot{\bm n}\}\|_{L^1(F)}+\|[\Gamma_{{\rm if}}-\Gamma_{{\rm if},h}^{{\rm DG}}]\|_{L^1(F)}\right)\nonumber\\
&+\|\nabla(\Gamma_{{\rm if}}-\Gamma_{{\rm if},h}^{{\rm DG}})\|_{L_h^1(\Omega)}\leq Ch^{\frac{d}{2}}\mathcal{K}_{\Omega_{\ast}(x_{\ast})}
+C\sum_{j=0}^Id_j^{\frac{d}{2}}\mathcal{K}_{\Omega_j(x_{\ast})}\label{rG23-L1-DG-proof:1}\\
\leq& Ch^{1+\frac{d}{2}}|\Gamma_{{\rm if}}|_{H^2(\Omega)}+C\sum_{j=0}^Id_j^{\frac{d}{2}}\mathcal{K}_{\Omega_j(x_{\ast})}\leq Ch+C\sum_{j=0}^Id_j^{\frac{d}{2}}\mathcal{K}_{\Omega_j(x_{\ast})},\nonumber
\end{align}
where
\begin{align*}
\mathcal{K}_{\Omega_j(x_{\ast})}=&\Big(\|\nabla(\Gamma_{{\rm if}}-\Gamma_{{\rm if},h}^{{\rm DG}})\|_{L_h^2(\Omega_j(x_{\ast}))}+
\Big(\sum_{F\in\mathcal{E}_h}\big(h_F\|\{\nabla(\Gamma_{{\rm if}}-\Gamma_{{\rm if},h}^{{\rm DG}})\cdot{\bm n}\}\|^2_{L^2(F\cap\Omega_j(x_{\ast}))}\\
&+h^{-1}_F\|[\Gamma_{{\rm if}}-\Gamma_{{\rm if},h}^{{\rm DG}}]\|^2_{L^2(F\cap\Omega_j(x_{\ast}))}\big)\Big)^{1/2}\Big).
\end{align*}

Using Lemma \ref{lee}, Lemma \ref{ee-rG23}, and interpolation error estimates, we have
\begin{align}
\sum_{j=0}^Id_j^{\frac{d}{2}}\mathcal{K}_{\Omega_j(x_{\ast})}\leq& C\sum_{j=0}^Ihd_j^{\frac{d}{2}}|\Gamma_{{\rm if}}|_{H^2(\Omega_j^2(x_{\ast}))}+C\sum_{j=0}^Id_j^{\frac{d}{2}-1}\|\Gamma_{{\rm if}}-\Gamma_{{\rm if},h}^{{\rm DG}}\|_{L^2(\Omega_j^1(x_{\ast}))}\nonumber\\
\leq& Ch\sum_{j=0}^Ihd_j^{-1}+C\sum_{j=0}^Id_j^{\frac{d}{2}-1}\|\Gamma_{{\rm if}}-\Gamma_{{\rm if},h}^{{\rm DG}}\|_{L^2(\Omega_j^1(x_{\ast}))}\label{rG23-L1-DG-proof:2}\\
\leq& Ch+C\sum_{j=0}^Id_j^{\frac{d}{2}-1}\|\Gamma_{{\rm if}}-\Gamma_{{\rm if},h}^{{\rm DG}}\|_{L^2(\Omega_j^1(x_{\ast}))}.\nonumber
\end{align}
We now estimate $\|\Gamma_{{\rm if}}-\Gamma_{{\rm if},h}^{{\rm DG}}\|_{L^2(\Omega_j^1(x_{\ast}))}$. To this end, let $W$ be the solution of
\begin{align*}
-\Delta W= f\quad {\rm in}~\Omega,\quad W=0\quad {\rm on}~\partial\Omega,
\end{align*}
for any $f\in\mathcal{C}_0^{\infty}(\Omega_j^1(x_{\ast}))$ with $\|f\|_{L^2(\Omega_j^1(x_{\ast}))}=1$. Then, the definition of $a_h$ and Lemma \ref{dfp} yield
\begin{align*}
(f,\Gamma_{{\rm if}}-\Gamma_{{\rm if},h}^{{\rm DG}})=a_h(W,\Gamma_{{\rm if}}-\Gamma_{{\rm if},h}^{{\rm DG}})=a_h(W-\Pi_k W,
\Gamma_{{\rm if}}-\Gamma_{{\rm if},h}^{{\rm DG}}).
\end{align*}
For the above equality, we estimate the right-hand side on the subdomains $\Omega_j^3(x_{\ast})$ and
$\Omega\backslash\Omega_j^3(x_{\ast})$ separately. On the subdomain $\Omega_j^3(x_{\ast})$, the interpolation error estimates yield
\begin{align*}
(f,\Gamma_{{\rm if}}-\Gamma_{{\rm if},h}^{{\rm DG}})|_{\Omega_j^3(x_{\ast})}\leq Ch\mathcal{K}_{\Omega_j^3(x_{\ast})}.
\end{align*}
On the subdomain $\Omega\backslash\Omega_j^3(x_{\ast})$, \cite[Lemma 3.8]{lc2026} and the interpolation error estimates presented in \cite[Lemma 4.1]{glrs2009} yield
\begin{align*}
&(f,\Gamma_{{\rm if}}-\Gamma_{{\rm if},h}^{{\rm DG}})|_{\Omega\backslash\Omega_j^3(x_{\ast})}\\
\leq &C\left(\|\nabla(W-\Pi_kW)\|_{L_h^{\infty}(\Omega\backslash\Omega_j^3(x_{\ast}))}+h^{-1}\|W-\Pi_kW\|_{L^{\infty}(\Omega\backslash\Omega_j^3(x_{\ast}))}\right)\times\\
&\Big(\sum_{F\in\mathcal{E}_h}\left(
\|[\Gamma_{{\rm if}}-\Gamma_{{\rm if},h}^{{\rm DG}}]\|_{L^1(F\cap\Omega\backslash\Omega_j^3(x_{\ast}))}+h_F\|\{\nabla(\Gamma_{{\rm if}}-\Gamma_{{\rm if},h}^{{\rm DG}})\cdot{\bm n}\}\|_{L^1(F\cap\Omega\backslash\Omega_j^3(x_{\ast}))}\right)\\
&+\|\nabla(\Gamma_{{\rm if}}-\Gamma_{{\rm if},h}^{{\rm DG}})\|_{L_h^1(\Omega\backslash\Omega_j^3(x_{\ast}))}\Big)\\
\leq&Ch^{\gamma}d_j^{1-\gamma-\frac{d}{2}}\Big(\sum_{F\in\mathcal{E}_h}\left(
\|[\Gamma_{{\rm if}}-\Gamma_{{\rm if},h}^{{\rm DG}}]\|_{L^1(F)}+h_F\|\{\nabla(\Gamma_{{\rm if}}-\Gamma_{{\rm if},h}^{{\rm DG}})\cdot{\bm n}\}\|_{L^1(F)}\right)\\
&+\|\nabla(\Gamma_{{\rm if}}-\Gamma_{{\rm if},h}^{{\rm DG}})\|_{L_h^1(\Omega)}\Big).
\end{align*}
Therefore, using the duality argument, Lemma \ref{leee}, and (\ref{ddd:2}), we obtain
\begin{align*}
&\sum_{j=0}^Id_j^{\frac{d}{2}-1}\|\Gamma_{{\rm if}}-\Gamma_{{\rm if},h}^{{\rm DG}}\|_{L^2(\Omega_j^1(x_{\ast}))}\\
\leq&Ch^{\frac{d}{2}}\mathcal{K}_{\Omega_{\ast}(x_{\ast})}+ C\sum_{j=0}^Ihd_j^{-1}d_j^{\frac{d}{2}}\mathcal{K}_{\Omega_j(x_{\ast})}+C\sum_{j=0}^I(hd_j^{-1})^{\gamma}\Big(\|\nabla(\Gamma_{{\rm if}}-\Gamma_{{\rm if},h}^{{\rm DG}})\|_{L_h^1(\Omega)}\\
&+\sum_{F\in\mathcal{E}_h}\left(\|[\Gamma_{{\rm if}}-\Gamma_{{\rm if},h}^{{\rm DG}}]\|_{L^1(F)}+h_F\|\{\nabla(\Gamma_{{\rm if}}-\Gamma_{{\rm if},h}^{{\rm DG}})\cdot{\bm n}\}\|_{L^1(F)}\right)\Big),
\end{align*}
which, together with (\ref{rG23-L1-DG-proof:1}) and (\ref{rG23-L1-DG-proof:2}), leads to the desired result (\ref{rG23-L1-DG}) by setting $M$ to be large enough.

Next, we prove the result (\ref{rG23-L1-HDG}) for the case of ${\rm s=if}$, since the case of ${\rm s=bf}$ can be proved similarly.
First of all, through the definitions of $\Pi_W$, $\Pi_V$, $\Pi_{\widehat{V}}$ and $\Pi_k$, we have the following error equation:
\begin{align}
(\Pi_W\boldsymbol{\mathcal{R}}_{{\rm if}}-\boldsymbol{\mathcal{R}}_{{\rm if},h},{\bm w}_h)_{\mathcal{T}_h}-
(\Pi_V\Gamma_{{\rm if}}-\Gamma_{{\rm if},h},\nabla\cdot{\bm w}_h)_{\mathcal{T}_h}+\langle\Pi_{\widehat{V}}\Gamma_{{\rm if}}-\widehat{\Gamma}_{{\rm if},h},{\bm w}_h\cdot{\bm n}\rangle_{\partial\mathcal{T}_h}=0,\label{rG23-L1-HDG-proof:1}
\end{align}
for any ${\bm w}_h\in {\bm W}_h$, and the following result:
\begin{align*}
h^{-1}_K\|\Gamma_{{\rm if},h}-\widehat{\Gamma}_{{\rm if},h}\|^2_{L^2(\partial K)}=&h_K^{-1}\langle(\Pi_{\widehat{V}}\Gamma_{{\rm if}}-\widehat{\Gamma}_{{\rm if},h})-(\Pi_V\Gamma_{{\rm if}}-\Gamma_{{\rm if},h}),\Gamma_{{\rm if},h}-\widehat{\Gamma}_{{\rm if},h}\rangle_{\partial K}\\
&+h_K^{-1}\langle\Pi_V\Gamma_{{\rm if}}-\Pi_{\widehat{V}}\Gamma_{{\rm if}},\Gamma_{{\rm if},h}-\widehat{\Gamma}_{{\rm if},h}\rangle_{\partial K}\\
\leq&Ch_K^{-1/2}\|\Gamma_{{\rm if},h}-\widehat{\Gamma}_{{\rm if},h}\|_{L^2(\partial K)}\Big(h_K^{-1/2}\|(\Pi_{\widehat{V}}\Gamma_{{\rm if}}-\widehat{\Gamma}_{{\rm if},h})-(\Pi_V\Gamma_{{\rm if}}-\Gamma_{{\rm if},h})\|_{L^2(\partial K)}\\
&+h^{-1}\|\Gamma_{{\rm if}}-\Pi_V\Gamma_{{\rm if}}\|_{L^2(K)}+\|\nabla(\Gamma_{{\rm if}}-\Pi_V\Gamma_{{\rm if}})\|_{L^2(K)}\\
&+h^{-1}\|\Gamma_{{\rm if}}-\Pi_k\Gamma_{{\rm if}}\|_{L^2(K)}+\|\nabla(\Gamma_{{\rm if}}-\Pi_k\Gamma_{{\rm if}})\|_{L^2(K)}\Big).
\end{align*}
Applying Lemma \ref{inequality} to (\ref{rG23-L1-HDG-proof:1}), the above inequality arrives at
\begin{align}
h^{-1/2}_K\|\Gamma_{{\rm if},h}-\widehat{\Gamma}_{{\rm if},h}\|_{L^2(\partial K)}\leq& C\Big(\|\Pi_W\boldsymbol{\mathcal{R}}_{{\rm if}}-\boldsymbol{\mathcal{R}}_{{\rm if},h}\|_{L^2(K)}+\|\sqrt{\tau}(\mathcal{L}\Gamma_{{\rm if},h}-\widehat{\Gamma}_{{\rm if},h})\|_{L^2(\partial K)}\nonumber\\
&+h^{-1}\|\Gamma_{{\rm if}}-\Pi_V\Gamma_{{\rm if}}\|_{L^2(K)}+\|\nabla(\Gamma_{{\rm if}}-\Pi_V\Gamma_{{\rm if}})\|_{L^2(K)}\label{rG23-L1-HDG-proof:2}\\
&+h^{-1}\|\Gamma_{{\rm if}}-\Pi_k\Gamma_{{\rm if}}\|_{L^2(K)}+\|\nabla(\Gamma_{{\rm if}}-\Pi_k\Gamma_{{\rm if}})\|_{L^2(K)}\Big).\nonumber
\end{align}
Then, using (\ref{ddd:2}), (\ref{rG23-L1-HDG-proof:2}), Lemma \ref{leee}, and interpolation error estimates, we obtain
\begin{align}
&\sum_{K\in\mathcal{T}_h}\|\Gamma_{{\rm if},h}-\widehat{\Gamma}_{{\rm if},h}\|_{L^1(\partial K)}+\|\boldsymbol{\mathcal{R}}_{{\rm if}}-\boldsymbol{\mathcal{R}}_{{\rm if},h}\|_{L^1(\Omega)}+\|\nabla(\Gamma_{{\rm if}}-\Gamma_{{\rm if},h})\|_{L_h^1(\Omega)}\nonumber\\
&+h\|\nabla(\boldsymbol{\mathcal{R}}_{{\rm if}}-\boldsymbol{\mathcal{R}}_{{\rm if},h})\|_{L_h^1(\Omega)}
\leq Ch^{\frac{d}{2}}\mathcal{J}_{\Omega_{\ast}(x_{\ast})}+C\sum_{j=0}^I d_j^{\frac{d}{2}}\mathcal{J}_{\Omega_j(x_{\ast})}\label{rG23-L1-HDG-proof:3}\\
\leq& Ch^{1+\frac{d}{2}}|\Gamma_{{\rm if}}|_{H^2(\Omega)}+C\sum_{j=0}^I d_j^{\frac{d}{2}}\mathcal{J}_{\Omega_j(x_{\ast})}\leq Ch+C\sum_{j=0}^I d_j^{\frac{d}{2}}\mathcal{J}_{\Omega_j(x_{\ast})},\nonumber
\end{align}
where
\begin{align*}
\mathcal{J}_{\Omega_j(x_{\ast})}=&\Big(\sum_{K\in\mathcal{T}_h}\|\sqrt{\tau}(\mathcal{L}\Gamma_{{\rm if},h}-\widehat{\Gamma}_{{\rm if},h})\|_{L^2(\partial K\cap\Omega_j(x_{\ast}))}^2\Big)^{1/2}+\|\boldsymbol{\mathcal{R}}_{{\rm if}}-\boldsymbol{\mathcal{R}}_{{\rm if},h}\|_{L^2(\Omega_j(x_{\ast}))}\\
&+\|\nabla(\Gamma_{{\rm if}}-\Gamma_{{\rm if},h})\|_{L_h^2(\Omega_j(x_{\ast}))}+h\|\nabla(\boldsymbol{\mathcal{R}}_{{\rm if}}-\boldsymbol{\mathcal{R}}_{{\rm if},h})\|_{L_h^2(\Omega_j(x_{\ast}))}+\|\boldsymbol{\mathcal{R}}_{{\rm if}}-\Pi_W\boldsymbol{\mathcal{R}}_{{\rm if}}\|_{L^2(\Omega_j(x_{\ast}))}\\
&+h^{-1}\|\Gamma_{{\rm if}}-\Pi_V\Gamma_{{\rm if}}\|_{L^2(\Omega_j(x_{\ast}))}+\|\nabla(\Gamma_{{\rm if}}-\Pi_V\Gamma_{{\rm if}})\|_{L^2(\Omega_j(x_{\ast}))}\\
&+h^{-1}\|\Gamma_{{\rm if}}-\Pi_k\Gamma_{{\rm if}}\|_{L^2(\Omega_j(x_{\ast}))}+\|\nabla(\Gamma_{{\rm if}}-\Pi_k\Gamma_{{\rm if}})\|_{L_h^2(\Omega_j(x_{\ast}))}.
\end{align*}
By virtue of Lemma \ref{lee}, Lemma \ref{ee-rG23}, and interpolation error estimates, we have
\begin{align}
\sum_{j=0}^Id_j^{\frac{d}{2}}\mathcal{J}_{\Omega_j(x_{\ast})}\leq& C\sum_{j=0}^Ihd_j^{\frac{d}{2}}|\Gamma_{{\rm if}}|_{H^2(\Omega_j^2(x_{\ast}))}
+C\sum_{j=0}^I(hd_j^{-1})d_j^{\frac{d}{2}}\mathcal{J}_{\Omega_j^1(x_{\ast})}\nonumber\\
&+C\sum_{j=0}^Id_j^{\frac{d}{2}-1}\|\Gamma_{{\rm if}}-\Gamma_{{\rm if},h}\|_{L^2(\Omega_j^1(x_{\ast}))}\label{rG23-L1-HDG-proof:4}\\
\leq& Ch+C\sum_{j=0}^I(hd_j^{-1})d_j^{\frac{d}{2}}\mathcal{J}_{\Omega_j(x_{\ast})}+C\sum_{j=0}^Id_j^{\frac{d}{2}-1}\|\Gamma_{{\rm if}}-\Gamma_{{\rm if},h}\|_{L^2(\Omega_j^1(x_{\ast}))}.\nonumber
\end{align}

To estimate $\|\Gamma_{{\rm if}}-\Gamma_{{\rm if},h}\|_{L^2(\Omega_j^1(x_{\ast}))}$, let $(\phi,\boldsymbol{\mathcal{A}})$ be the solution of
\begin{align*}
\boldsymbol{\mathcal{A}}+\nabla \phi=0\quad {\rm in}~\Omega,\quad \nabla\cdot\boldsymbol{\mathcal{A}}=f\quad {\rm in}~\Omega,\quad \phi=0\quad {\rm on}~\partial\Omega,
\end{align*}
for any $f\in \mathcal{C}_0^{\infty}(\Omega_j^1(x_{\ast}))$ with $\|f\|_{L^2(\Omega_j^1(x_{\ast}))}=1$. Then, the definition of $\mathcal{B}$ and Lemma \ref{dfp} yield
\begin{align*}
-(f,\Gamma_{{\rm if}}-\Gamma_{{\rm if},h})=&\mathcal{B}(\boldsymbol{\mathcal{A}},\phi,\phi;\boldsymbol{\mathcal{R}}_{{\rm if}}-\boldsymbol{\mathcal{R}}_{{\rm if},h},\Gamma_{{\rm if}}-\Gamma_{{\rm if},h},\Gamma_{{\rm if}}-\widehat{\Gamma}_{{\rm if},h})+\langle\tau(\mathcal{L}\phi-\phi),\mathcal{L}\Gamma_{{\rm if}}-\Gamma_{{\rm if}}\rangle_{\partial\mathcal{T}_h}\\
=&\mathcal{B}(\boldsymbol{\mathcal{A}}-\Pi_W\boldsymbol{\mathcal{A}},\phi-\Pi_V\phi,\phi-\Pi_{\widehat{V}}\phi;\boldsymbol{\mathcal{R}}_{{\rm if}}-\boldsymbol{\mathcal{R}}_{{\rm if},h},
\Gamma_{{\rm if}}-\Gamma_{{\rm if},h},\Gamma_{{\rm if}}-\widehat{\Gamma}_{{\rm if},h})\\
&+\langle\tau(\mathcal{L}\phi-\phi),\mathcal{L}\Gamma_{{\rm if}}-\Gamma_{{\rm if}}\rangle_{\partial\mathcal{T}_h}\\
=&(\boldsymbol{\mathcal{A}}-\Pi_W\boldsymbol{\mathcal{A}},\boldsymbol{\mathcal{R}}_{{\rm if}}-\boldsymbol{\mathcal{R}}_{{\rm if},h})+
(\nabla(\phi-\Pi_V\phi),\boldsymbol{\mathcal{R}}_{{\rm if}}-\boldsymbol{\mathcal{R}}_{{\rm if},h})_{\mathcal{T}_h}\\
&-\langle\phi-\Pi_V\phi,(\boldsymbol{\mathcal{R}}_{{\rm if}}-\boldsymbol{\mathcal{R}}_{{\rm if},h})\cdot{\bm n}\rangle_{\partial\mathcal{T}_h}+
(\boldsymbol{\mathcal{A}}-\Pi_W\boldsymbol{\mathcal{A}},\nabla(\Gamma_{{\rm if}}-\Gamma_{{\rm if},h}))_{\mathcal{T}_h}\\
&+\langle(\boldsymbol{\mathcal{A}}-\Pi_W\boldsymbol{\mathcal{A}})\cdot{\bm n},\Gamma_{{\rm if},h}-\widehat{\Gamma}_{{\rm if},h}\rangle_{\partial\mathcal{T}_h}
+\langle\tau(\Pi_{\widehat{V}}\phi-\mathcal{L}\Pi_V\phi),\mathcal{L}\Gamma_{{\rm if},h}-\widehat{\Gamma}_{{\rm if},h}\rangle_{\partial\mathcal{T}_h}
\end{align*}
Similar to the IPDG method, we estimate the right-hand side of the above equality on subdomains $\Omega_j^3(x_{\ast})$ and $\Omega\backslash\Omega_j^3(x_{\ast})$ separately. On subdomain $\Omega_j^3(x_{\ast})$, the interpolation error estimates and (\ref{rG23-L1-HDG-proof:2}) yield
\begin{align*}
-(f,\Gamma_{{\rm if}}-\Gamma_{{\rm if},h})|_{\Omega_j^3(x_{\ast})}\leq Ch\mathcal{J}_{\Omega_j^4(x_{\ast})}.
\end{align*}
On subdomain $\Omega\backslash\Omega_j^3(x_{\ast})$, \cite[Lemma 3.8]{lc2026} and the interpolation error estimates presented in \cite[Lemma 4.1]{glrs2009} yield
\begin{align*}
-(f,\Gamma_{{\rm if}}-\Gamma_{{\rm if},h})|_{\Omega\backslash\Omega_j^3(x_{\ast})}\leq& Ch^{\gamma}d_j^{1-\gamma-\frac{d}{2}}\Big(\|\boldsymbol{\mathcal{R}}_{{\rm if}}-\boldsymbol{\mathcal{R}}_{{\rm if},h}\|_{L^1(\Omega)}+h\|\nabla(\boldsymbol{\mathcal{R}}_{{\rm if}}-\boldsymbol{\mathcal{R}}_{{\rm if},h})\|_{L_h^1(\Omega)}\\
&+\|\nabla(\Gamma_{{\rm if}}-\Gamma_{{\rm if},h})\|_{L_h^1(\Omega)}+\sum_{K\in\mathcal{T}_h}\|\Gamma_{{\rm if},h}-\widehat{\Gamma}_{{\rm if},h}\|_{L^1(\partial K)}\Big).
\end{align*}
Thus, using the duality argument, Lemma \ref{leee}, and (\ref{ddd:2}), we obtain
\begin{align*}
&\sum_{j=0}^Id_j^{\frac{d}{2}-1}\|\Gamma_{{\rm if}}-\Gamma_{{\rm if},h}\|_{L^2(\Omega_j^1(x_{\ast}))}\\
\leq& C\sum_{j=0}^I(hd_j^{-1})^{\gamma}\Big(\|\boldsymbol{\mathcal{R}}_{{\rm if}}-\boldsymbol{\mathcal{R}}_{{\rm if},h}\|_{L^1(\Omega)}+h\|\nabla(\boldsymbol{\mathcal{R}}_{{\rm if}}-\boldsymbol{\mathcal{R}}_{{\rm if},h})\|_{L_h^1(\Omega)}+\|\nabla(\Gamma_{{\rm if}}-\Gamma_{{\rm if},h})\|_{L_h^1(\Omega)}\\
&+\sum_{K\in\mathcal{T}_h}\|\Gamma_{{\rm if},h}-\widehat{\Gamma}_{{\rm if},h}\|_{L^1(\partial K)}\Big)+Ch^{\frac{d}{2}}\mathcal{J}_{\Omega_j(x_{\ast})}+C\sum_{j=0}^I(hd_j^{-1})d_j^{\frac{d}{2}}\mathcal{J}
_{\Omega_j(x_{\ast})},
\end{align*}
which, together with (\ref{rG23-L1-HDG-proof:3}) and (\ref{rG23-L1-HDG-proof:4}), derives the desired result (\ref{rG23-L1-HDG}) by setting
$M$ to be large enough.
\end{proof}

\subsection{Proof of Corollary \ref{mn-HDG} and Corollary \ref{mn-HDG-1}}\label{sec3:5}
Firstly, we provide the proof of Corollary \ref{mn-HDG}.
Since $\tau=0$ for the hybrid mixed DG method, the result (\ref{mn-HDG:1}) naturally holds for this method.
For the HDG-1 and HDG-2 methods, we know from (\ref{HDG:3}) that on each $F\in\mathcal{E}_h^o$
\begin{align*}
\tau^+(\mathcal{L}u_h^+-\widehat{u}_h)+\tau^-(\mathcal{L}u_h^--\widehat{u}_h)=&-[{\bm q}_h\cdot{\bm n}],\\
(\mathcal{L}u_h^+-\widehat{u}_h)-(\mathcal{L}u_h^--\widehat{u}_h)=&\mathcal{L}[u_h],
\end{align*}
where $v^+$ and $v^-$ denote the restrictions of $v$ on $\partial K^+\cap F$ and $\partial K^-\cap F$ for $v=u_h,\tau$, and $K^+$ and $K^-$ denote the elements sharing the common face $F$. Then, by a simple calculation, we derive
\begin{align*}
(\tau^++\tau^-)(\mathcal{L}u_h^+-\widehat{u}_h)=-[{\bm q}_h\cdot{\bm n}]+\tau^-\mathcal{L}[u_h],\\
(\tau^++\tau^-)(\mathcal{L}u_h^--\widehat{u}_h)=-[{\bm q}_h\cdot{\bm n}]-\tau^+\mathcal{L}[u_h].
\end{align*}
Therefore, for each $K\in\mathcal{T}_h$,
\begin{align*}
\|\tau(\mathcal{L}u_h-\widehat{u}_h)\|_{L^{\infty}(\partial K)}\leq& \|[{\bm q}_h\cdot{\bm n}]\|_{L^{\infty}(\partial K\backslash\mathcal{E}_h^{\partial})}+Ch_K^{-1}\|[u_h]\|_{L^{\infty}(\partial K)}\\
\leq &\|{\bm q}-{\bm q}_h\|_{L^{\infty}(\Omega)}+Ch_K^{-1}\|[u_h]\|_{L^{\infty}(\partial K)},
\end{align*}
which, together with Theorem \ref{mn-jump} and Theorem \ref{mn-g}, leads to the desired result (\ref{mn-HDG:1}) for the HDG-1 and HDG-2 methods. The proof of Corollary \ref{mn-HDG} is complete.

For Corollary \ref{mn-HDG-1}, using (\ref{rG23-L1-HDG-proof:2}) and the inverse estimate, we have for each $K\in\mathcal{T}_h$
\begin{align}
h_K^{-1}\|u_h-\widehat{u}_h\|_{L^{\infty}(\partial K)}\leq& Ch_K^{-\frac{d+1}{2}}\|u_h-\widehat{u}_h\|_{L^2(\partial K)}\leq Ch_K^{-\frac{d}{2}}\Big(
\|\Pi_W{\bm q}-{\bm q}_h\|_{L^2(K)}\nonumber\\
&+\|\sqrt{\tau}(\mathcal{L}u_h-\widehat{u}_h)\|_{L^2(\partial K)}+h_K^{-1}\|u-\Pi_Vu\|_{L^2(K)}+\|\nabla(u-\Pi_Vu)\|_{L^2(K)}\nonumber\\
&+h_K^{-1}\|u-\Pi_ku\|_{L^2(K)}+\|\nabla(u-\Pi_ku)\|_{L^2(K)}\Big)\nonumber\\
\leq&C\Big(\|{\bm q}-\Pi_W{\bm q}\|_{L^{\infty}(K)}+\|{\bm q}-{\bm q}_h\|_{L^{\infty}(K)}+\|\tau(\mathcal{L}u_h-\widehat{u}_h)\|_{L^{\infty}(\partial K)}\nonumber\\
&+\|\nabla(u-\Pi_Vu)\|_{L^{\infty}(K)}+\|\nabla(u-\Pi_ku)\|_{L^{\infty}(K)}\Big),\nonumber
\end{align}
which, together with Theorem \ref{mn-g} and Corollary \ref{mn-HDG}, results in
\begin{align}
h_K^{-1}\|u_h-\widehat{u}_h\|_{L^{\infty}(\partial K)}\leq&C\Big(\|{\bm q}-\Pi_W{\bm q}\|_{L^{\infty}(\Omega)}+
\|\nabla(u-\Pi_Vu)\|_{L_h^{\infty}(\Omega)}\nonumber\\
&+\|\nabla(u-\Pi_ku)\|_{L_h^{\infty}(\Omega)}\Big),\label{mn-HDG-1-proof:1}
\end{align}
For the HDG-2 method, the above estimate (\ref{mn-HDG-1-proof:1}) is the result (\ref{mn-HDG-1:2}) in Corollary \ref{mn-HDG-1}. For the hybrid mixed DG method and the HDG-1 method, the above estimate (\ref{mn-HDG-1-proof:1}) is the result (\ref{mn-HDG-1:1}) in Corollary \ref{mn-HDG-1} since $\Pi_k=\Pi_V$ for these two cases. Therefore, the proof of Corollary \ref{mn-HDG-1} is complete.

\section{Application to the biharmonic problem}\label{sec5}
One attractive feature of the interface Green's function framework is that it is independent of the particular
DG formulation. To demonstrate this generality, we apply it to the \(C^0\) interior penalty method for the
biharmonic equation in convex polygonal domains.

In this section, we assume that $\Omega$ is a two-dimensional bounded and convex
polygonal domain. We consider the following biharmonic equation:
\begin{align}
\Delta^2u=f\quad {\rm in}~\Omega,\quad u=\nabla u\cdot{\bm n}=0\quad {\rm on}~\partial\Omega.\label{model-bi}
\end{align}
Here $u$ is the unknown and $f$ is the source term.
Let $D^s$ denote the general $s$-th order differential operator. The weak formulation of (\ref{model-bi}) then reads as follows: Find
$u\in H_0^2(\Omega)=\{v\in H^2(\Omega): v=\nabla v\cdot{\bm n}=0,~{\rm on}~\partial\Omega\}$ such that
\begin{align}
(D^2u,D^2v)=(f,v),\quad \forall v\in H_0^2(\Omega).
\end{align}
From \cite[Sec. 5.9]{k1967}, we know that for each $f\in H^{-1}(\Omega)$, the above weak formulation has a unique weak solution $u\in H_0^2(\Omega)\cap H^3(\Omega)$ satisfying
\begin{align}
\|u\|_{H^3(\Omega)}\leq C\|f\|_{H^{-1}(\Omega)}.\label{bi-regularity}
\end{align}
In addition, similar to (\ref{model_assumptions}), we assume that there is $s>0$ such that the solution $u$ of the fourth-order problem (\ref{model-bi}) has the following regularity:
\begin{align}
u\in W^{2+s,\infty}(\Omega).\label{model-bi-assumptions}
\end{align}
According to the embedding theorem (cf. \cite[Theorem~$1.4.4.1$]{g1985}), (\ref{model-bi-assumptions}) implies that $u\in\mathcal{C}^{2,\alpha}(\overline{\Omega})$ with $\alpha\in(0,1)$ depending on $s$.

\subsection{$C^0$ interior penalty methods}\label{sec4:1}
Based on simplicial mesh $\mathcal{T}_h$, we define the following discrete space:
\begin{align*}
S_h=\{v_h\in H_0^1(\Omega): v_h|_K\in \mathcal{P}^k(K),~\forall K\in\mathcal{T}_h\},\quad (k\geq 2).
\end{align*}
For each edge $F\in\mathcal{E}_h^o$ shared by elements $K^+$ and $K^-$, we define the jump of the normal derivative and the average of the second normal derivative of $v$ across $F$ as follows:
\begin{align*}
\left[\nabla v\cdot{\bm n}_F\right]=\left[\frac{\partial v}{\partial {\bm n}_F}\right]=\frac{\partial v^+}{\partial {\bm n}_F}-\frac{\partial v^-}{\partial {\bm n}_F}\quad {\rm and}\quad \left\{{\bm n}_F^T(D^2v){\bm n}_F\right\}=\left\{\frac{\partial^2 v}{\partial{\bm n}_F^2}\right\}=\frac{1}{2}\left(
\frac{\partial^2 v^+}{\partial {\bm n}_F^2}+\frac{\partial^2 v^-}{\partial{\bm n}_F^2}\right),
\end{align*}
where ${\bm n}_F$ denotes the unit normal pointing from $K^-$ to $K^+$.
For each edge $F\in\mathcal{E}_h^{\partial}$, we take ${\bm n}_F$ to be the unit normal pointing outside $\Omega$ and define
\begin{align*}
\left[\frac{\partial v}{\partial {\bm n}_F}\right]=-\frac{\partial v}{\partial {\bm n}_F}\quad {\rm and}\quad \left\{\frac{\partial^2 v}{\partial{\bm n}_F^2}\right\}=\frac{\partial^2 v}{\partial{\bm n}_F^2}.
\end{align*}
Then the $C^0$ interior penalty ($C^{0}$-IP) method for the fourth order problem (\ref{model-bi}) is to find $u_h\in S_h$ such that
\begin{align}
A_h(u_h,v_h)=(f,v_h),\quad \forall v_h\in S_h,\label{IPG}
\end{align}
where the bilinear form $A_h$ is defined as
\begin{align*}
A_h(u_h,v_h)= & (D^2u_h,D^2v_h)_{\mathcal{T}_h}
+\sum_{F\in\mathcal{E}_h}\int_F \left( \left\{\frac{\partial^2 u_h}{\partial{\bm n}_F^2}\right\}\left[
\frac{\partial v_h}{\partial {\bm n}_F}\right]+\left\{\frac{\partial^2 v_h}{\partial{\bm n}_F^2}\right\}\left[
\frac{\partial u_h}{\partial {\bm n}_F}\right] \right) ds \\
& \quad + \sum_{F\in \mathcal{E}_{h}}\int_{F} \frac{\sigma}{h_F}
\left[\frac{\partial u_h}{\partial {\bm n}_F}\right]\left[\frac{\partial v_h}{\partial {\bm n}_F}\right] ds,
\end{align*}
and $\sigma>0$ is a (sufficiently large) stabilization parameter independent of the mesh size $h$.
Since the bilinear form $A_h$ is bounded and coercive (cf. \cite{bs2005}), the discrete problem (\ref{IPG})
has a unique solution.

\subsection{Main results}\label{sec4:2}
In \cite{l2021}, Leykekhman has proved the following maximum-norm error estimate for the $C^0$-IP method for the
biharmonic equation on two dimensional convex polygonal domains:
\begin{align}
\|D^2(u-u_h)\|_{L^{\infty}_h(\Omega)}\leq C|\ln h|^{\frac{3}{2}}\left(h^{-1}\|\nabla(u-I_hu)\|_{L^{\infty}(\Omega)}+\|D^2(u-I_hu)\|_{L_h^{\infty}(\Omega)}\right).
\end{align}
For the $C^0$-IP method, $\|D^2(u-u_h)\|_{L^{\infty}_h(\Omega)}$ is also only a seminorm. Therefore, to bridge this gap, using a similar method that has been established in Section \ref{sec3}, we prove the following results. Note that their proofs are provided in Section \ref{sec4:3}.
\begin{theorem}\label{mn-bi-jump}
Let $u$ and $u_h$ be the solutions of problems (\ref{model-bi}) and (\ref{IPG}).

(1) We assume (\ref{model-bi-assumptions}) holds, then
\begin{align}
\max_{F\in\mathcal{E}_h}h^{-1}_F\left\|\left[\frac{\partial(u-u_h)}{\partial {\bm n}_F}\right]\right\|_{L^{\infty}(F)}\leq& C|\ln h|^{\frac{3}{2}}\Big(h^{-1}\|\nabla(u-I_hu)\|_{L^{\infty}(\Omega)}\\
&+\|D^2(u-I_hu)\|_{L_h^{\infty}(\Omega)}\Big).\nonumber
\end{align}

(2) If the largest inner angle of the polygonal domain is less than $126.28^\circ$, we further have
\begin{align}
\max_{F\in\mathcal{E}_h}h^{-1}_F\left\|\left[\frac{\partial(u-u_h)}{\partial {\bm n}_F}\right]\right\|_{L^{\infty}(F)}\leq& C|\ln h|\Big(h^{-1}\|\nabla(u-I_hu)\|_{L^{\infty}(\Omega)}\\
&+\|D^2(u-I_hu)\|_{L_h^{\infty}(\Omega)}\Big).\nonumber
\end{align}
Here the positive constants $C$ are independent of the mesh size $h$ and $I_h$ denotes the standard Lagrange interpolation
operator onto $S_h$.
\end{theorem}

In \cite[Theorem 4.2]{hltwz2026}, if the largest inner angle of the polygonal domain is less than $126.28^\circ$, the authors established the following local error estimate: For any $z\in\Omega$,
\begin{align}
|\nabla(u-u_h)(z)|\leq& C\left(\|\nabla(u-I_hu)\|_{L^{\infty}(B_{5h})}+h^{-1}\|u-I_hu\|_{L^{\infty}(B_{5h})}\right)\nonumber\\
&+C|\ln h|^{\frac{1}{2}}|||(u-I_hu)|||_{2,B_{\theta},z}+Ch^2\theta^{-2}|\ln h|^{\frac{1}{2}}|||u-I_hu|||_{2,\Omega},\label{pe-bi:2}
\end{align}
for $k\geq 3$,
where $B_s=B(z,s)\cap\Omega$, $B(z,s)$ is a ball of radius $s$ centered at $z$, and
\begin{align*}
|||v|||_{2,\mathcal{D}}^2=&\sum_{K\in\mathcal{T}_h(\mathcal{D})}\|D^2v\|^2_{L^2(K)}+\sum_{F\in\mathcal{E}_h\cap\mathcal{T}_h(\mathcal{D})}h_F
\left\|\left\{\frac{\partial^2v}{\partial{\bm n}_F^2}\right\}\right\|^2_{L^2(F)}+h_F^{-1}\left\|\left[\frac{\partial v}{\partial{\bm n}_F}\right]\right\|^2_{L^2(F)},\\
|||v|||_{2,\mathcal{D},z}^2=&\sum_{K\in\mathcal{T}_h(\mathcal{D})}\left(\frac{h^\frac{3}{2}}{{\rm dist}(z,K)^{\frac{3}{2}}+h^{\frac{3}{2}}}|||v|||_{2,K}\right)^2,
\end{align*}
for $\mathcal{D}\subset\Omega$ and $\mathcal{T}_h(\mathcal{D})=\{K\in\mathcal{T}_h: K\cap\mathcal{D}\neq\emptyset\}$.
According to the proof of Theorem 4.1 and Theorem 4.2 in \cite{hltwz2026}, we know that $\theta$ is a positive constant satisfying $\theta\geq h$. Therefore,
using the $L^{\infty}$ norm to estimate the last term in (\ref{pe-bi:2}), we obtain
\begin{align*}
h^2\theta^{-2}|\ln h|^{\frac{1}{2}}|||u-I_hu|||_{2,\Omega}\leq C|\ln h|^{\frac{1}{2}}\left(h^{-1}\|\nabla(u-I_hu)\|_{L^{\infty}(\Omega)}+
\|D^2(u-I_hu)\|_{L_h^{\infty}(\Omega)}\right).
\end{align*}
Then, for each $F\in\mathcal{E}_h$, we only can derive from (\ref{pe-bi:2}) and the above estimate the following suboptimal result:
\begin{align*}
h_F^{-1}\left\|\left[\frac{\partial(u-u_h)}{\partial {\bm n}_F}\right]\right\|_{L^{\infty}(F)}\leq& Ch^{-1}|\ln h|^{\frac{1}{2}}\left(h^{-1}\|\nabla(u-I_hu)\|_{L^{\infty}(\Omega)}+\|D^2(u-I_hu)\|_{L_h^{\infty}(\Omega)}\right).
\end{align*}
Compared with the above result, our results presented in Theorem \ref{mn-bi-jump} are much better.

\section{Conclusion}\label{sec6}
In this paper, we have developed an interface Green's function framework to establish a complete discrete $W^{1,\infty}$ analysis for discontinuous and nonconforming finite element methods. Historically, obtaining optimal, logarithm-free maximum-norm estimates for interface jumps for polynomial degrees strictly greater than one on convex polyhedral domains has been obstructed by the lack of weak discrete maximum principles for DG methods in convex polyhedral domains. To overcome this fundamental barrier, we introduced a purely methodological innovation: rather than relying on standard Green's functions driven by a single discrete delta function, we utilized Green's functions generated by the difference between adjacent discrete delta functions.

By exploiting integral property of the discrete delta function, we explicitly captured the cancellation mechanism between these neighboring singular sources. This structural cancellation elegantly neutralizes the dominant singular behavior, yielding a novel, sharper local energy estimate. Consequently, our framework completely circumvents the need for a weak discrete maximum principle and effectively removes the suboptimal logarithmic factors inherited from classical trace-inequality-based approaches.

The power and versatility of the proposed framework have been extensively demonstrated. We successfully established optimal, complete discrete $W^{1,\infty}$ error estimates simultaneously bounding broken gradients and interface jumps for the symmetric IPDG, the hybrid mixed DG, and two representative HDG methods applied to the Poisson equation in convex polyhedral domains. Furthermore, our successful extension of this framework to the $C^0$ interior penalty approximation of the biharmonic equation in convex polygonal domains, which yielded the first maximum-norm bounds for the jumps of the normal derivative, confirms that our analytical techniques are robust and extend naturally beyond DG methods for the Poisson problem.

\smallskip
\smallskip

\appendix

\section{}

\subsection{Proof of Theorem \ref{mn-g}}\label{sec3:4}

In the following, we provide the complete proof of Theorem \ref{mn-g}.
For each $x_0\in K_{x_0}$ with $K_{x_0}\in\mathcal{T}_h$, let the regularized Green's function $\Gamma(x_0,x)$
be the solution of
\begin{align}
-\Delta\Gamma(x_0,x)=\nabla\cdot({\bm e}_i\delta_{h,x_0})\quad {\rm in}~\Omega,\quad \Gamma(x_0,x)=0\quad {\rm on}~\partial\Omega,\label{rG1}
\end{align}
where ${\bm e}_i$ is the $i$-th standard basis vector in $\mathbb{R}^d$. By setting $\boldsymbol{\mathcal{R}}_1(x_0,x)=-\nabla\Gamma(x_0,x)$ and $\boldsymbol{\mathcal{ R}}_2
(x_0,x)=-\nabla\Gamma(x_0,x)-{\bm e}_i\delta_{h,x_0}$, we arrive at two mixed formulations of the above problem
\begin{align}
\boldsymbol{\mathcal{R}}_1+\nabla\Gamma=0\quad {\rm in}~\Omega,\quad
\nabla\cdot\boldsymbol{\mathcal{R}}_1=\nabla\cdot({\bm e}_i\delta_{h,x_0})\quad {\rm in}~\Omega,\quad
\Gamma=0\quad {\rm on}~\partial\Omega,\label{rG1-1}
\end{align}
and
\begin{align}
\boldsymbol{\mathcal{R}}_2+\nabla\Gamma=-{\bm e}_i\delta_{h,x_0}\quad {\rm in}~\Omega,\quad
\nabla\cdot\boldsymbol{\mathcal{R}}_2=0\quad {\rm in}~\Omega,\quad
\Gamma=0\quad {\rm on}~\partial\Omega.\label{rG1-2}
\end{align}

\begin{lemma}\label{rG1-lee}
Let $\Gamma(x_0,x)$ be the regularized Green's function defined in (\ref{rG1}), then we have
\begin{align}
|\Gamma|_{H^2(\Omega_j^2(x_0))}\leq C d_j^{-1-\frac{d}{2}}.
\end{align}
\end{lemma}
\begin{proof}
The result presented in this lemma can be found in \cite[Lemma 3.7]{lc2026} and \cite[(4.10)]{dlsw2012}.
\end{proof}

Let $\Gamma_h^{{\rm DG}}\in V_h^{{\rm DG}}$ be the IPDG approximation of problem (\ref{rG1}) such that
\begin{align}
a_h(\Gamma_h^{{\rm DG}},v_h)=(v_h,\nabla\cdot({\bm e}_i\delta_{h,x_0})),\quad \forall v_h\in V_h^{{\rm DG}},\label{rG1-DG}
\end{align}
and let $(\boldsymbol{\mathcal{R}}_{1,h},\Gamma_{1,h},\widehat{\Gamma}_{1,h})\in {\bm W}_h\times V_h\times \widehat{V}_h$ and $(\boldsymbol{\mathcal{R}}_{2,h},\Gamma_{2,h},\widehat{\Gamma}_{2,h})\in {\bm W}_h\times V_h\times \widehat{V}_h$ be the HDG approximations of problems
(\ref{rG1-1}) and (\ref{rG1-2}) satisfying
\begin{align}
\mathcal{B}(\boldsymbol{\mathcal{R}}_{1,h},\Gamma_{1,h},\widehat{\Gamma}_{1,h};{\bm w}_h,v_h,\widehat{v}_h)=-(\nabla\cdot({\bm e}_i\delta_{h,x_0}),v_h),\quad \forall ({\bm w}_h,v_h,\widehat{v}_h)\in {\bm W}_h\times V_h\times \widehat{V}_h,\label{rG1-1-HDG}
\end{align}
and
\begin{align}
\mathcal{B}(\boldsymbol{\mathcal{R}}_{2,h},\Gamma_{2,h},\widehat{\Gamma}_{2,h};{\bm w}_h,v_h,\widehat{v}_h)=-({\bm e}_i\delta_{h,x_0},{\bm w}_h),\quad \forall ({\bm w}_h,v_h,\widehat{v}_h)\in{\bm W}_h\times V_h\times \widehat{V}_h.\label{rG1-2-HDG}
\end{align}
Then, similar to the derivation of (\ref{rG2:1}), we have by using Lemma \ref{dfp}, (\ref{ddd:2}), (\ref{rG1}), and (\ref{rG1-DG})
\begin{align}
&-({\bm e}_i\cdot\nabla(\Pi_ku-u_h^{{\rm DG}}))(x_0)=(\Pi_ku-u_h^{{\rm DG}}, \nabla\cdot({\bm e}_i\delta_{h,x_0}))\nonumber\\
=&a_h(\Gamma_h^{{\rm DG}}-\Gamma,\Pi_ku-u)+(\nabla\cdot({\bm e}_i\delta_{h,x_0}),\Pi_ku-u)_{\mathcal{T}_h}\nonumber\\
\leq&C\Big(\sum_{F\in\mathcal{E}_h}\left(h_F\|\{\nabla(\Gamma-\Gamma_h^{{\rm DG}})\cdot{\bm n}\}\|_{L^1(F)}+\|[\Gamma-\Gamma_h^{{\rm DG}}]\|_{L^1(F)}\right)\nonumber\\
&+\|\nabla(\Gamma-\Gamma_h^{{\rm DG}})\|_{L_h^1(\Omega)}\Big)\Big(\|\nabla(u-\Pi_ku)\|_{L_h^{\infty}(\Omega)}+h^{-1}\|u-\Pi_ku\|_{L^{\infty}(\Omega)}\Big)\label{rG1:1}\\
&+\|\nabla\cdot({\bm e}_i\delta_{h,x_0})\|_{L^1(\Omega)}\|u-\Pi_ku\|_{L^{\infty}(\Omega)}\nonumber\\
\leq&C\Big(1+\sum_{F\in\mathcal{E}_h}\left(h_F\|\{\nabla(\Gamma-\Gamma_h^{{\rm DG}})\cdot{\bm n}\}\|_{L^1(F)}+\|[\Gamma-\Gamma_h^{{\rm DG}}]\|_{L^1(F)}\right)\nonumber\\
&+\|\nabla(\Gamma-\Gamma_h^{{\rm DG}})\|_{L_h^1(\Omega)}\Big)\|\nabla(u-\Pi_ku)\|_{L_h^{\infty}(\Omega)}.\nonumber
\end{align}
Similar to the derivation of (\ref{rG2:2}), we obtain from Lemma \ref{dfp}, (\ref{ddd:2}), (\ref{rG1-1}), (\ref{rG1-2}), (\ref{rG1-1-HDG}) and (\ref{rG1-2-HDG}) that
\begin{align}
&\left({\bm e}_i\cdot\nabla(\Pi_Vu-u_h)\right)(x_0)=-(\Pi_Vu-u_h,\nabla\cdot({\bm e}_i\delta_{h,x_0}))_{K_{x_0}}\nonumber\\
=&\mathcal{B}(\boldsymbol{\mathcal{R}}_{1,h}-\boldsymbol{\mathcal{R}}_1,\Gamma_{1,h}-\Gamma,\widehat{\Gamma}_{1,h}-\Gamma;\Pi_W{\bm q}-{\bm q},\Pi_Vu-u,\Pi_{\widehat{V}}u-u)\nonumber\\
&+\mathcal{B}(\boldsymbol{\mathcal{R}}_1,\Gamma,\Gamma;\Pi_W{\bm q}-{\bm q},\Pi_Vu-u,\Pi_{\widehat{V}}u-u)\nonumber\\
=&(\boldsymbol{\mathcal{R}}_{1,h}-\boldsymbol{\mathcal{R}}_1,\Pi_W{\bm q}-{\bm q})+(\nabla(\Gamma_{1,h}-\Gamma),\Pi_W{\bm q}-{\bm q})_{\mathcal{T}_h}\nonumber\\
&-\langle\Gamma_{1,h}-\widehat{\Gamma}_{1,h},(\Pi_W{\bm q}-{\bm q})\cdot{\bm n}\rangle_{\partial\mathcal{T}_h}+
(\boldsymbol{\mathcal{R}}_{1,h}-\boldsymbol{\mathcal{R}}_1,\nabla(\Pi_Vu-u))_{\mathcal{T}_h}\label{rG1-1:1}\\
&-\langle(\boldsymbol{\mathcal{R}}_{1,h}-\boldsymbol{\mathcal{R}}_1)\cdot{\bm n},\Pi_Vu-u\rangle_{\partial\mathcal{T}_h}-(\nabla\cdot({\bm e}_i\delta_{h,x_0}),
\Pi_Vu-u)_{\mathcal{T}_h}\nonumber\\
&-\langle\tau(\mathcal{L}\Gamma_{1,h}-\widehat{\Gamma}_{1,h}),\mathcal{L}(\Pi_Vu-u)\rangle_{\partial\mathcal{T}_h}\nonumber\\
\leq&C\Big(1+\|\boldsymbol{\mathcal{R}}_{1,h}-\boldsymbol{\mathcal{R}}_1\|_{L^1(\Omega)}+\|\nabla(\Gamma_{1,h}-\Gamma)\|_{L_h^1(\Omega)}
+h\|\nabla(\boldsymbol{\mathcal{R}}_{1,h}-\boldsymbol{\mathcal{R}}_1)\|_{L_h^1(\Omega)}\nonumber\\
&+\sum_{K\in\mathcal{T}_h}\|\Gamma_{1,h}-\widehat{\Gamma}_{1,h}\|_{L^1(\partial K)}\Big)\Big(\|{\bm q}-\Pi_W{\bm q}\|_{L^{\infty}(\Omega)}+
\|\nabla(u-\Pi_Vu)\|_{L^{\infty}_h(\Omega)}\Big),\nonumber
\end{align}
and
\begin{align}
&-((\Pi_W{\bm q}-{\bm q}_h)\cdot{\bm e}_i)(x_0)=-(\Pi_W{\bm q}-{\bm q}_h,{\bm e}_i\delta_{h,x_0})\nonumber\\
=&\mathcal{B}(\boldsymbol{\mathcal{R}}_{2,h}-\boldsymbol{\mathcal{R}}_2,\Gamma_{2,h}-\Gamma,\widehat{\Gamma}_{2,h}-\Gamma;\Pi_W{\bm q}-{\bm q},\Pi_Vu-u,\Pi_{\widehat{V}}u-u)\nonumber\\
&+\mathcal{B}(\boldsymbol{\mathcal{R}}_2,\Gamma,\Gamma;\Pi_W{\bm q}-{\bm q},\Pi_Vu-u,\Pi_{\widehat{V}}u-u)\nonumber\\
\leq& C\Big(1+\|\boldsymbol{\mathcal{R}}_{2,h}-\boldsymbol{\mathcal{R}}_2\|_{L^1(\Omega)}+\|\nabla(\Gamma_{2,h}-\Gamma)\|_{L_h^1(\Omega)}
+h\|\nabla(\boldsymbol{\mathcal{R}}_{2,h}-\boldsymbol{\mathcal{R}}_2)\|_{L_h^1(\Omega)}\label{rG1-2:1}\\
&+\sum_{K\in\mathcal{T}_h}\|\Gamma_{2,h}-\widehat{\Gamma}_{2,h}\|_{L^1(\partial K)}\Big)\Big(\|{\bm q}-\Pi_W{\bm q}\|_{L^{\infty}(\Omega)}+
\|\nabla(u-\Pi_Vu)\|_{L^{\infty}_h(\Omega)}\Big).\nonumber
\end{align}

Next, we prove Lemma~\ref{rG1-L1} below, which, together with (\ref{rG1:1}), (\ref{rG1-1:1}),  (\ref{rG1-2:1}) and the fact that $(\Pi_{k}u-u,1)_{K}=(\Pi_{V}u-u,1)_{K}=0$ for any $K \in \mathcal{T}_{h}$, leads
to the desired results presented in Theorem~\ref{mn-g}.
\begin{lemma}\label{rG1-L1}
Let $(\boldsymbol{\mathcal{R}}_1,\Gamma)$ and $(\boldsymbol{\mathcal{R}}_2,\Gamma)$ be the solutions of problems (\ref{rG1-1}) and (\ref{rG1-2}). Let
$\Gamma_h^{{\rm DG}}$, $(\boldsymbol{\mathcal{R}}_{1,h},\Gamma_{1,h},\widehat{\Gamma}_{1,h})$ and $(\boldsymbol{\mathcal{R}}_{2,h},\Gamma_{2,h},
\widehat{\Gamma}_{2,h})$ be their IPDG and HDG approximations defined in (\ref{rG1-DG}), (\ref{rG1-1-HDG}) and (\ref{rG1-2-HDG}) respectively. Then, we have
\begin{align}
\sum_{F\in\mathcal{E}_h}\left(h_F\|\{\nabla(\Gamma-\Gamma_h^{{\rm DG}})\cdot{\bm n}\}\|_{L^1(F)}+\|[\Gamma-\Gamma_h^{{\rm DG}}]\|_{L^1(F)}\right)
+\|\nabla(\Gamma-\Gamma_h^{{\rm DG}})\|_{L_h^1(\Omega)}\leq C,
\end{align}
and
\begin{align}
&\sum_{K\in\mathcal{T}_h}\|\Gamma_{{\rm s},h}-\widehat{\Gamma}_{{\rm s},h}\|_{L^1(\partial K)}+\|\boldsymbol{\mathcal{R}}_{{\rm s},h}-\boldsymbol{\mathcal{R}}_{\rm s}\|_{L^1(\Omega)}\nonumber\\
&+\|\nabla(\Gamma_{{\rm s},h}-\Gamma)\|_{L_h^1(\Omega)}+h\|\nabla(\boldsymbol{\mathcal{R}}_{{\rm s},h}-\boldsymbol{\mathcal{R}}_{\rm s})\|_{L_h^1(\Omega)}
\leq C,
\end{align}
for ${\rm s\in\{1,2\}}$.
\end{lemma}
\begin{proof}
Since the proof of this lemma is identical to that of Lemma \ref{rG23-L1}, except that Lemma \ref{ee-rG23} is replaced by Lemma \ref{rG1-lee}, we omit the details.
\end{proof}

\subsection{Auxiliary results of $C^{0}$-IP method for the biharmonic equation}

Next, we introduce some well known results associated to (\ref{IPG}).
\begin{lemma}\label{dfp-bi}
Let $u$ and $u_h$ be the solutions of problems (\ref{model-bi}) and (\ref{IPG}), then we have
\begin{align}
A_h(u,v_h)=(f,v_h),\quad \forall v_h\in S_h,
\end{align}
and
\begin{align}
&\|D^2(u-u_h)\|_{L_h^2(\Omega)}+\left(\sum_{F\in\mathcal{E}_h}h_F\left\|\left\{\frac{\partial^2(u-u_h)}{\partial{\bm n}_F^2}\right\}\right\|^2_{L^2(F)}
+h_F^{-1}\left\|\left[\frac{\partial(u-u_h)}{\partial{\bm n}_F}\right]\right\|^2_{L^2(F)}\right)^{1/2}\\
\leq&\|D^2(u-I_hu)\|_{L_h^2(\Omega)}+\left(\sum_{F\in\mathcal{E}_h}h_F\left\|\left\{\frac{\partial^2(u-I_hu)}{\partial{\bm n}_F^2}\right\}\right\|^2_{L^2(F)}
+h_F^{-1}\left\|\left[\frac{\partial(u-I_hu)}{\partial{\bm n}_F}\right]\right\|^2_{L^2(F)}\right)^{1/2},\nonumber
\end{align}
where $I_h$ denotes the standard Lagrange interpolation operator onto $S_h$.
\end{lemma}
\begin{proof}
These results follow from \cite[Lemma 5 and Lemma 8]{bs2005}.
\end{proof}
\begin{lemma}\label{lee-bi}
Let $u$ and $u_h$ be the solutions of problems (\ref{model-bi}) and (\ref{IPG}), and let $\Omega_0\subset\Omega_1\subset\Omega$ satisfy
$\rho={\rm dist}(\partial\Omega_0,\partial\Omega_1\backslash\partial\Omega)\geq rh$ for $r>1$ sufficiently large. Then we have
\begin{align}
&\left(\sum_{F\in\mathcal{E}_h}h_F\left\|\left\{\frac{\partial^2(u-u_h)}{\partial{\bm n}_F^2}\right\}\right\|^2_{L^2(F\cap\Omega_0)}
+h_F^{-1}\left\|\left[\frac{\partial(u-u_h)}{\partial{\bm n}_F}\right]\right\|^2_{L^2(F\cap\Omega_0)}\right)^{1/2}\nonumber\\
&+\|D^2(u-u_h)\|_{L^2_h(\Omega_0)}\leq C\rho^{-2}\left(\|u-I_hu\|_{L^2(\Omega_1)}+\|u-u_h\|_{L^2(\Omega_1)}\right)\\
&\quad+ C\|D^2(u-I_hu)\|_{L^2_h(\Omega_1)}+C\Big(\sum_{F\in\mathcal{E}_h}h_F\left\|\left\{\frac{\partial^2(u-I_hu)}{\partial{\bm n}_F^2}\right\}\right\|^2_{L^2(F\cap\Omega_1)}\nonumber\\
&\quad \quad\quad\quad+h_F^{-1}\left\|\left[\frac{\partial(u-I_hu)}{\partial{\bm n}_F}\right]\right\|^2_{L^2(F\cap\Omega_1)}\Big)^{1/2}.\nonumber
\end{align}
\end{lemma}
\begin{proof}
This result can be found in \cite[Lemma 3.9]{l2021}.
\end{proof}

\subsection{Proof of Theorem \ref{mn-bi-jump}}\label{sec4:3}
For $z\in\Omega$, in section \ref{sec4:3} we let $G(z,x)$ be the Green's function associated with (\ref{model-bi}), which is defined as
\begin{align*}
\Delta^2G(z,x)=\delta_z\quad {\rm in}~\Omega,\quad G(z,x)=\nabla G(z,x)\cdot{\bm n}=0\quad {\rm on}~\partial\Omega.
\end{align*}
From \cite[Lemma 2.1]{l2021}, we know that the Green's function $G(z,x)$ satisfies the following pointwise estimates:
\begin{align}
|D_z^{\alpha}D_x^{\beta}G(z,x)|\leq C|x-z|^{2-|\alpha|-|\beta|},\quad \forall 1\leq |\alpha|,|\beta|\leq 2,~3\leq |\alpha|+|\beta|\leq 4,\label{pe-G-bi}
\end{align}
for $x\neq z$.

Similar to Section \ref{sec3:3}, for each point $x_{\ast}\in F$ with $F\in\mathcal{E}_h^o$ shared by the elements $K^+_{\ast}$ and $K^-_{\ast}$,
we define the following problem
\begin{align}
\Delta^2g_{{\rm if}}(x_{\ast},x)=\nabla\cdot({\bm e}_i\delta_{h,K^+_{\ast}})-\nabla\cdot({\bm e}_i\delta_{h,K^-_{\ast}})\quad {\rm in}~\Omega,\quad
g_{{\rm if}}=\nabla g_{{\rm if}}\cdot{\bm n}=0\quad {\rm on}~\partial\Omega,\label{rG-bi}
\end{align}
and let $g_{{\rm if},h}\in S_h$ be its $C^0$-IP approximation such that
\begin{align}
A_h(g_{{\rm if},h},v_h)=(\nabla\cdot({\bm e}_i\delta_{h,K^+_{\ast}})-\nabla\cdot({\bm e}_i\delta_{h,K^-_{\ast}}),v_h)_{\mathcal{T}_h},\quad \forall v_h\in S_h.\label{rG-bi-IPG}
\end{align}
Then, using (\ref{ddd:2}), Lemma \ref{dfp-bi}, (\ref{rG-bi}) and (\ref{rG-bi-IPG}), we obtain
\begin{align}
&(\nabla (I_hu-u_h)\cdot{\bm e}_i)|_{K^-_{\ast}}(x_{\ast})-(\nabla (I_hu-u_h)\cdot{\bm e}_i)|_{K^+_{\ast}}(x_{\ast})\nonumber\\
=&(I_hu-u_h,\nabla\cdot({\bm e}_i\delta_{h,K^+_{\ast}})-\nabla\cdot({\bm e}_i\delta_{h,K^-_{\ast}}))_{\mathcal{T}_h}\nonumber\\
=&A_h(g_{{\rm if},h}-g_{{\rm if}},I_hu-u)+A_h(g_{{\rm if}},I_hu-u)\nonumber\\
\leq&\Big(\sum_{F\in\mathcal{E}_h}\Big(h_F\left\|\left\{\frac{\partial^2(g_{{\rm if},h}-g_{{\rm if}})}{\partial{\bm n}_F^2}\right\}\right\|_{L^1(F)}+\left\|\left[\frac{\partial(g_{{\rm if},h}-g_{{\rm if}})}{\partial{\bm n}_F}\right]\right\|_{L^1(F)}\Big)\label{rG-bi:1}\\
&+\|D^2(g_{{\rm if},h}-g_{{\rm if}})\|_{L_h^1(\Omega)}\Big)\Big(\|D^2(I_hu-u)\|_{L_h^{\infty}(\Omega)}+h^{-1}\|\nabla(I_hu-u)\|_{L^{\infty}(\Omega)}\Big)\nonumber\\
&+(\nabla\cdot({\bm e}_i\delta_{h,K^+_{\ast}})-\nabla\cdot({\bm e}_i\delta_{h,K^-_{\ast}}),I_hu-u)\nonumber\\
\leq&\Big(Ch+\sum_{F\in\mathcal{E}_h}\Big(h_F\left\|\left\{\frac{\partial^2(g_{{\rm if},h}-g_{{\rm if}})}{\partial{\bm n}_F^2}\right\}\right\|_{L^1(F)}+\left\|\left[\frac{\partial(g_{{\rm if},h}-g_{{\rm if}})}{\partial{\bm n}_F}\right]\right\|_{L^1(F)}\Big)\nonumber\\
&+\|D^2(g_{{\rm if},h}-g_{{\rm if}})\|_{L_h^1(\Omega)}\Big)\Big(\|D^2(I_hu-u)\|_{L_h^{\infty}(\Omega)}+h^{-1}\|\nabla(I_hu-u)\|_{L^{\infty}(\Omega)}\Big).\nonumber
\end{align}

Next, for each point $x^{\ast}\in F$ with $F\in \mathcal{E}_h^{\partial}$ included in the element $K^{\ast}$, we define the following problem
\begin{align}
\Delta^2g_{{\rm bf}}(x^{\ast},x)=\nabla\cdot({\bm e}_i\delta_{h,x^{\ast}})\quad {\rm in}~\Omega,\quad g_{{\rm bf}}=\nabla g_{{\rm bf}}\cdot{\bm n}=0\quad {\rm on}~\partial\Omega,\label{rG1-bi}
\end{align}
and let $g_{{\rm bf},h}\in S_h$ be its $C^0$-IP approximation such that
\begin{align}
A_h(g_{{\rm bf},h},v_h)=(\nabla\cdot({\bm e}_i\delta_{h,x^{\ast}}),v_h),\quad \forall v_h\in S_h.\label{rG1-bi-IPG}
\end{align}
Then, similar to (\ref{rG-bi:1}), we have by using (\ref{ddd:2}), Lemma \ref{dfp-bi}, (\ref{rG1-bi}) and (\ref{rG1-bi-IPG})
\begin{align}
&-(\nabla (I_hu-u_h)\cdot{\bm e}_i)|_{K^{\ast}}(x^{\ast})=(I_hu-u_h,\nabla\cdot({\bm e}_i\delta_{h,x^{\ast}}))\nonumber\\
=&A_h(g_{{\rm bf},h}-g_{{\rm bf}},I_hu-u)+A_h(g_{{\rm bf}},I_hu-u)\label{rG1-bi:1}\\
\leq&\Big(Ch+\sum_{F\in\mathcal{E}_h}\Big(h_F\left\|\left\{\frac{\partial^2(g_{{\rm bf},h}-g_{{\rm bf}})}{\partial{\bm n}_F^2}\right\}\right\|_{L^1(F)}+\left\|\left[\frac{\partial(g_{{\rm bf},h}-g_{{\rm bf}})}{\partial{\bm n}_F}\right]\right\|_{L^1(F)}\Big)\nonumber\\
&+\|D^2(g_{{\rm bf},h}-g_{{\rm bf}})\|_{L_h^1(\Omega)}\Big)\Big(\|D^2(I_hu-u)\|_{L_h^{\infty}(\Omega)}+h^{-1}\|\nabla(I_hu-u)\|_{L^{\infty}(\Omega)}\Big).\nonumber
\end{align}

Before we provide a further estimate for (\ref{rG-bi:1}) and (\ref{rG1-bi:1}), we prove the following local energy estimates for $g_{{\rm if}}$ and $g_{{\rm bf}}$.
\begin{lemma}\label{rG-bi-lee}
Let $g_{{\rm if}}$ and $g_{{\rm bf}}$ be the solutions of problems (\ref{rG-bi}) and (\ref{rG1-bi}), then we have
\begin{align}
|D^2g_{{\rm if}}|_{L^{\infty}(\Omega_j^2(x_{\ast}))}+|D^2g_{{\rm bf}}|_{L^{\infty}(\Omega_j^2(x^{\ast}))}\leq Chd_j^{-2},
\end{align}
and
\begin{align}
|D^2g_{{\rm if}}|_{L^2(\Omega_j^2(x_{\ast}))}+|D^2g_{{\rm bf}}|_{L^2(\Omega_j^2(x^{\ast}))}\leq Chd_j^{-1}.
\end{align}
\end{lemma}
\begin{proof}
For any $x\in \Omega_j^2(x_{\ast})$, using the definitions of the Green's function $G$ and the Dirac delta function $\delta_x$ we have
\begin{align}
g_{{\rm if}}(x_{\ast},x)=&(\delta_x(y),g_{{\rm if}}(x_{\ast},y))=(D^2G(x,y),D^2g_{{\rm if}}(x_{\ast},y))\nonumber\\
=&(\nabla\cdot({\bm e}_i\delta_{h,K^+_{\ast}})-\nabla\cdot({\bm e}_i\delta_{h,K^-_{\ast}}),G(x,y))_{\mathcal{T}_h}\nonumber\\
=&({\bm e}_i\delta_{h,K^-_{\ast}},\nabla G(x,y))_{K^-_{\ast}}-({\bm e}_i\delta_{h,K^+_{\ast}},\nabla G(x,y))_{K^+_{\ast}}.\nonumber
\end{align}
Therefore,
\begin{align}
\partial_{x_ix_j} g_{{\rm if}}(x_{\ast},x)=&\int_{K^-_{\ast}}({\bm e}_i\delta_{h,K^-_{\ast}})(y)\cdot(\partial_{x_ix_j}\nabla_y G(x,y)) dy\label{rG-bi-lee-proof:1}\\
&-\int_{K^+_{\ast}}({\bm e}_i\delta_{h,K^+_{\ast}})(y)\cdot(\partial_{x_ix_j}\nabla_y G(x,y)) dy.\nonumber
\end{align}
Since $\partial_{x_ix_j}\nabla_y G(x,y)$, as a function of the second variable $y$, has the first order derivative in $B(x_{\ast},2h)\cap\Omega$ (see (\ref{pe-G-bi})), using Taylor's expansion
we know that there are two constants $t_3,t_4\in (0,1)$ depending on $y$ such that
\begin{align*}
\partial_{x_ix_j}\nabla_y G(x,y)=\partial_{x_ix_j}\nabla_y G(x,x_{\ast})+\partial_{x_ix_j}D^2_y G(x,y+t_3(x_{\ast}-y))(y-x_{\ast}),\quad \forall y\in K_{\ast}^-,\\
\partial_{x_ix_j}\nabla_y G(x,y)=\partial_{x_ix_j}\nabla_y G(x,x_{\ast})+\partial_{x_ix_j}D^2_y G(x,y+t_4(x_{\ast}-y))(y-x_{\ast}),\quad \forall y\in K_{\ast}^+,
\end{align*}
where $y+t_3(x_{\ast}-y)$ for $y\in K_{\ast}^+$ and $y+t_4(x_{\ast}-y)$ for $y\in K_{\ast}^-$ belong to $B(x_{\ast},2h)\cap\Omega$ since the domain $\Omega$ is convex. Moreover, according to (\ref{ddd:1}) we have
\begin{align*}
({\bm e}_i\delta_{h,K_{\ast}^+}(y),\partial_{x_ix_j}\nabla_y G(x,x_{\ast}))_{K_{\ast}^+}=({\bm e}_i\delta_{h,K_{\ast}^-}(y),\partial_{x_ix_j}\nabla_y G(x,x_{\ast}))_{K_{\ast}^-}.
\end{align*}
 Then, through (\ref{pe-G-bi}), (\ref{rG-bi-lee-proof:1}) and the above three identities, we obtain
\begin{align*}
\left|\partial_{x_ix_j} g_{{\rm if}}(x_{\ast},x)\right|\leq&\int_{K^-_{\ast}}\left|\delta_{h,K^-_{\ast}}\right|\left|\partial_{x_ix_j}D^2_y G(x,y+t_3(x_{\ast}-y))\right|\left|y-x_{\ast}\right| dy\\
&+\int_{K^+_{\ast}}\left|\delta_{h,K^+_{\ast}}\right|\left|\partial_{x_ix_j}D^2_y G(x,y+t_4(x_{\ast}-y))\right|\left|y-x_{\ast}\right| dy\\
\leq&C\int_{K^-_{\ast}}h\left|\delta_{h,K^-_{\ast}}\right|\left|t_3(x_{\ast}-x)+(1-t_3)(y-x)\right|^{-2} dy\\
&+C\int_{K^+_{\ast}}h\left|\delta_{h,K^+_{\ast}}\right|\left|t_4(x_{\ast}-x)+(1-t_4)(y-x)\right|^{-2} dy\leq Chd_j^{-2},
\end{align*}
and
\begin{align*}
|D^2 g_{{\rm if}}|_{L^2(\Omega_j^2(x_{\ast}))}\leq Chd_j^{-1}.
\end{align*}

For any $x\in\Omega_j^2(x^{\ast})$, similar to (\ref{rG-bi-lee-proof:1}), we have
\begin{align}
\partial_{x_ix_j}g_{{\rm bf}}(x^{\ast},x)=-\int_{K^{\ast}}({\bm e}_i\delta_{h,x^{\ast}})(y)\cdot(\partial_{x_ix_j}\nabla_yG(x,y)) dy.\label{rG-bi-lee-proof:2}
\end{align}
Since $\partial_{x_ix_j}\nabla_yG(x,y)$, as a function of the second variable $y$, has the first order derivative in $B(x^{\ast},2h)\cap\Omega$ (see (\ref{pe-G-bi})),
Taylor's expansion, together with $\partial_{x_ix_j}\nabla_yG(x,x^{\ast})=0$ for any $x\in\Omega_j^2(x^{\ast})$, implies that there exists a constant
$m_2\in(0,1)$ depending on $y$ such that
\begin{align*}
\partial_{x_ix_j}\nabla_yG(x,y)=\partial_{x_ix_j}D_y^2G(x,y+m_2(x^{\ast}-y))(y-x^{\ast}),\quad \forall y\in K^{\ast},
\end{align*}
where $y+m_2(x^{\ast}-y)\in B(x^{\ast},2h)\cap\Omega$ since the domain $\Omega$ is convex.
Thus, (\ref{pe-G-bi}), (\ref{rG-bi-lee-proof:2}) and the above identity yield
\begin{align*}
\left|\partial_{x_ix_j}g_{{\rm bf}}(x^{\ast},x)\right|\leq&\int_{K^{\ast}}\left|\delta_{h,x^{\ast}}\right|\left|\partial_{x_ix_j}D_y^2G(x,y+m_2(x^{\ast}-y))\right|\left|y-x^{\ast}\right|dy\\
\leq&C\int_{K^{\ast}}h\left|\delta_{h,x^{\ast}}\right|\left|m_2(x-x^{\ast})+(1-m_2)(x-y)\right|^{-2} dy\leq Chd_j^{-2},
\end{align*}
and
\begin{align*}
|D^2 g_{{\rm bf}}|_{L^2(\Omega_j^2(x^{\ast}))}\leq Chd_j^{-1}.
\end{align*}
We conclude the proof.
\end{proof}

Finally, based on Lemma \ref{rG-bi-lee}, we prove the following lemma, which, together with (\ref{rG-bi:1}) and (\ref{rG1-bi:1}), yields the results
presented in Theorem \ref{mn-bi-jump}.
\begin{lemma}\label{rG-bi-L1}
Let $g_{{\rm if}}$ and $g_{{\rm bf}}$ be the solutions of problems (\ref{rG-bi}) and (\ref{rG1-bi}), and let $g_{{\rm if},h}$ and $g_{{\rm bf},h}$ be their
$C^0$-IP approximations defined in (\ref{rG-bi-IPG}) and (\ref{rG1-bi-IPG}). Then, for ${\rm s=if}$ and ${\rm bf}$ we have
\begin{align}
\sum_{F\in\mathcal{E}_h}\Big(h_F&\left\|\left\{\frac{\partial^2(g_{{\rm s},h}-g_{{\rm s}})}{\partial{\bm n}_F^2}\right\}\right\|_{L^1(F)}+\left\|\left[\frac{\partial(g_{{\rm s},h}-g_{{\rm s}})}{\partial{\bm n}_F}\right]\right\|_{L^1(F)}\Big)\nonumber\\
&+\|D^2(g_{{\rm s},h}-g_{{\rm s}})\|_{L_h^1(\Omega)}\leq Ch|\ln h|^{\frac{3}{2}}.
\end{align}
If the largest inner angle of the polygonal domain is less than ${\rm 126.28^\circ}$, we have
\begin{align}
\sum_{F\in\mathcal{E}_h}\Big(h_F&\left\|\left\{\frac{\partial^2(g_{{\rm s},h}-g_{{\rm s}})}{\partial{\bm n}_F^2}\right\}\right\|_{L^1(F)}+\left\|\left[\frac{\partial(g_{{\rm s},h}-g_{{\rm s}})}{\partial{\bm n}_F}\right]\right\|_{L^1(F)}\Big)\nonumber\\
&+\|D^2(g_{{\rm s},h}-g_{{\rm s}})\|_{L_h^1(\Omega)}\leq Ch|\ln h|.
\end{align}
\end{lemma}
\begin{proof}
Here we only prove the case of ${\rm s=if}$, since the case of ${\rm s=bf}$ can be proved similarly. Using (\ref{ddd:2}) and Lemma \ref{dfp-bi}, we obtain
\begin{align}
&\sum_{F\in\mathcal{E}_h}\Big(h_F\left\|\left\{\frac{\partial^2(g_{{\rm if},h}-g_{{\rm if}})}{\partial{\bm n}_F^2}\right\}\right\|_{L^1(F)}+\left\|\left[\frac{\partial(g_{{\rm if},h}-g_{{\rm if}})}{\partial{\bm n}_F}\right]\right\|_{L^1(F)}\Big)
+\|D^2(g_{{\rm if},h}-g_{{\rm if}})\|_{L_h^1(\Omega)}\nonumber\\
\leq& Ch\mathcal{F}_{\Omega_{\ast}(x_{\ast})}+\sum_{j=0}^Id_j\mathcal{F}_{\Omega_j(x_{\ast})}\leq Ch^2|D^3 g_{{\rm if}}|_{L^2(\Omega)}+\sum_{j=0}^Id_j\mathcal{F}_{\Omega_j(x_{\ast})}\leq Ch+\sum_{j=0}^Id_j\mathcal{F}_{\Omega_j(x_{\ast})},\label{rG-bi-L1-proof:1}
\end{align}
where
\begin{align*}
\mathcal{F}_{\Omega_j(x_{\ast})}=&\left(\sum_{F\in\mathcal{E}_h}\Big(h_F\left\|\left\{\frac{\partial^2(g_{{\rm if},h}-g_{{\rm if}})}{\partial{\bm n}_F^2}\right\}\right\|_{L^2(F\cap\Omega_j(x_{\ast}))}^2+h_F^{-1}\left\|\left[\frac{\partial(g_{{\rm if},h}-g_{{\rm if}})}{\partial{\bm n}_F}\right]\right\|_{L^2(F\cap\Omega_j(x_{\ast}))}^2\Big)\right)^{1/2}\\
&+\|D^2(g_{{\rm if},h}-g_{{\rm if}})\|_{L_h^2(\Omega_j(x_{\ast}))}.
\end{align*}
Applying Lemma \ref{lee-bi}, Lemma \ref{rG-bi-lee}, and interpolation error estimates, we have
\begin{align*}
\mathcal{F}_{\Omega_j(x_{\ast})}^2\leq& C\left(\sum_{F\in\mathcal{E}_h\cap\Omega_j^1(x_{\ast})}h_F^2\right)\|D^2(g_{{\rm if}}-I_hg_{{\rm if}})\|_{L_h^{\infty}(\Omega_j^1(x_{\ast}))}^2\\
&+C(1+h^4d_j^{-4})|D^2 g_{{\rm if}}|_{L^2(\Omega_j^2(x_{\ast}))}^2+Cd_j^{-4}\|g_{{\rm if}}-g_{{\rm if},h}\|^2_{L^2(\Omega_j^1(x_{\ast}))}\\
\leq&Ch^2d_j^{-2}+Cd_j^{-4}\|g_{{\rm if}}-g_{{\rm if},h}\|^2_{L^2(\Omega_j^1(x_{\ast}))},
\end{align*}
and this leads to
\begin{align}
\sum_{j=0}^Id_j\mathcal{F}_{\Omega_j(x_{\ast})}\leq& Ch\sum_{j=0}^I 1+C\sum_{j=0}^Id_j^{-1}\|g_{{\rm if}}-g_{{\rm if},h}\|_{L^2(\Omega_j^1(x_{\ast}))}\label{rG-bi-L1-proof:2}\\
\leq& Ch|\ln h|+C\sum_{j=0}^Id_j^{-1}\|g_{{\rm if}}-g_{{\rm if},h}\|_{L^2(\Omega_j^1(x_{\ast}))}.\nonumber
\end{align}

Furthermore, using the duality argument presented in \cite[Subsection 4.2]{l2021}, we have
\begin{align}
\sum_{j=0}^Id_j^{-1}\|g_{{\rm if}}-g_{{\rm if},h}\|_{L^2(\Omega_j^1(x_{\ast}))}\leq Ch|\ln h|^{\frac{3}{2}}.\label{rG-bi-L1-proof:3}
\end{align}
If the largest inner angle of the polygonal domain is less than $126.28^\circ$, to estimate $\|g_{{\rm if}}-g_{{\rm if},h}\|_{L^2_h(\Omega_j^1(x_{\ast}))}$, we
let $W$ be the solution of the following fourth order problem:
\begin{align*}
\Delta^2 W=f\quad {\rm in}~\Omega,\quad W=\nabla W\cdot{\bm n}=0\quad {\rm on}~\partial\Omega,
\end{align*}
for $f\in\mathcal{C}_0^{\infty}(\Omega_j^1(x_{\ast}))$ with $\|f\|_{L^2(\Omega_j^1(x_{\ast}))}=1$. From \cite[Theorem 2]{br1980} and \cite[(7)]{hltwz2026}, we know that
the above problem has a unique weak solution $W\in H_0^2(\Omega)\cap H^4(\Omega)$ satisfying
\begin{align*}
\|W\|_{H^4(\Omega)}\leq C\|f\|_{L^2(\Omega)}\leq C.
\end{align*}
Using (\ref{ddd:2}), Lemma \ref{dfp-bi}, and interpolation error estimates, we obtain
\begin{align*}
(f,g_{{\rm if}}-g_{{\rm if},h})=&A_h(W,g_{{\rm if}}-g_{{\rm if},h})=A_h(W-I_hW,g_{{\rm if}}-g_{{\rm if},h})\\
\leq&Ch^4\|W\|_{H^4(\Omega)}\|g_{{\rm if}}\|_{H^4(\Omega)}\\
\leq& Ch^4\|W\|_{H^4(\Omega)}\|\nabla\cdot({\bm e}_i\delta_{h,K_{\ast}^+})-\nabla\cdot({\bm e}_i\delta_{h,K^-_{\ast}})\|_{L^2(\Omega)}\leq Ch^2.
\end{align*}
Thus,
\begin{align}
\sum_{j=0}^Id_j^{-1}\|g_{{\rm if}}-g_{{\rm if},h}\|_{L^2(\Omega_j^1(x_{\ast}))}\leq Ch\sum_{j=0}^Ihd_j^{-1}\leq Ch,\label{rG-bi-L1-proof:4}
\end{align}
for the case that the largest inner angle of the polygonal domain is less than $126.28^\circ$.
We conclude the proof by using (\ref{rG-bi-L1-proof:1}), (\ref{rG-bi-L1-proof:2}), (\ref{rG-bi-L1-proof:3}), and (\ref{rG-bi-L1-proof:4}).
\end{proof}

{\bf Acknowledgements}  All authors contributed equally. Haitao Leng was supported by Basic and
Applied Basic Research Foundation of Guangdong Province (Grant No. 2024A1515011163, 2025A1515010428).


\end{document}